\documentclass[11pt]{amsart}
\usepackage{geometry}                
\usepackage{graphicx}
\usepackage{amssymb}
\usepackage[usenames,dvipsnames]{color}
\usepackage{hyperref}
\usepackage{cleveref}
\usepackage{float}

\theoremstyle{plain}
\newtheorem{theorem}{Theorem}[section]

\newtheorem{proposition}{Proposition}[section]
\newtheorem{corollary}{Corollary}[section]

\theoremstyle{definition}
\newtheorem{definition}{Definition}[section]
\newtheorem{remark}{Remark}[section]

\newcommand{\re}[1]{\textrm{Re}\left( #1 \right)}
\newcommand{\imag}[1]{\textrm{Im}\left( #1 \right)}
\newcommand{\sige}{\sigma_{\textrm{c}}}
\newcommand{\epsi}{\delta}
\newcommand{\ve}{\varepsilon}

\newcommand{\A}{\mathbb{A}}
\newcommand{\C}{\mathbb{C}}

\newcommand{\I}{\mathbb{I}}
\newcommand{\R}{\mathbb{R}}

\newcommand{\cA}{\mathcal{A}}
\newcommand{\cB}{\mathcal{B}}
\newcommand{\cC}{\mathcal{C}}

\newcommand{\cH}{\mathcal{H}}
\newcommand{\cL}{\mathcal{L}}
\newcommand{\cR}{\mathcal{R}}
\newcommand{\cT}{\mathcal{T}}

\newcommand{\rd}{\textrm{d}}

\newcommand{\bv}{{\bf{v}}}
\newcommand{\bw}{{\bf{w}}}

\newcommand{\bW}{{\bf{W}}}

\title[A Geometrically Inspired Evans Function]{Numerical computation of an Evans function for travelling waves}

\begin{document}
\author[K. Harley]{K. Harley$^1$}\thanks{$^1$Mathematical Sciences School, Queensland University of Technology, Brisbane, QLD 4000 Australia}
\author[P. van Heijster]{P. van Heijster$^1$}
\author[R. Marangell]{R. Marangell$^{2\dagger}$}\thanks{$^2$School of Mathematics and Statistics, University of Sydney, Sydney, NSW 2006 Australia}
\author[G. J. Pettet]{G. J. Pettet $^1$}\thanks{$^\dagger$Corresponding Author: \textsc{email:} robert.marangell@sydney.edu.au}
\author[M. Wechselberger]{M. Wechselberger$^2$}

\begin{abstract}
We demonstrate a geometrically inspired technique for computing Evans functions for the linearised operators about travelling waves. Using the examples of the F-KPP equation and a Keller-Segel model of bacterial chemotaxis, we produce an Evans function which is computable through several orders of magnitude in the spectral parameter and show how such a function can naturally be extended into the continuous spectrum. In both examples, we use this function to numerically verify the absence of eigenvalues in a large region of the right half of the spectral plane. We also include a new proof of spectral stability in the appropriate weighted space of travelling waves of speed $c \geq 2 \sqrt{\delta}$ in the F-KPP equation. 
\end{abstract}

\maketitle

\section{Introduction}\label{sec:intro} 

The main focus of this article is a geometrically inspired technique for numerically analysing the spectral stability of a travelling wave. In particular, we illustrate a new method for computing  
an {\em Evans function} for a linearised operator, linearised about a travelling wave solution to a partial differential equation (PDE) in 1+1 independent variables. Evans functions first arose in the 1970's \cite{evans75} and are typically constructed using a matching condition (see, for example, \cite{agj90, garzum98, sand02}). They are analytic functions in some relevant region of the complex $\lambda$-spectral plane, with the property that the multiplicity of their roots coincide with the multiplicity of $\lambda$ as an eigenvalue. A well-known obstacle in the numerical computation of the Evans function is the tendency of the {associated eigenvalue} ordinary differential equations (ODEs) to become stiff, an issue that can sometimes be overcome by working in the exterior product space \cite{allenbridges02, brigesderksgott02, brin98, brinzum02}. Building on this, there have been two complementary directions to the study of (efficient) numerical computation of the Evans function: Continuous Orthogonalization (see \cite{humzum06, humstanzum06, zum09} and the references therein for details) and Grassmannian Spectral Shooting, or the {\em Riccati approach} \cite{beckmal13, ledouxmalniesenthum08, ledouxmalthum10}. 

It is this second point of view that we follow here. By following the Riccati approach, and using the geometric structure of the problem, we are led to an Evans function whose computation is efficiently tractable even through relatively large changes in the order of the spectral parameter. 
The geometric interpretation of the Riccati equation allows us 
to avoid the problem of finite time blowup of the relevant solutions to nonlinear ODEs. Furthermore, this approach is readily extendible to values of the spectral parameter which lie in the continuous spectrum (also called the essential spectrum). In extending the Evans function into the continuous spectrum, we highlight another underlying relationship between instabilities of a travelling wave and the geometry of the spectral problem.

{\color{black}
Our approach shows the power of topological/dynamical systems techniques to analyse analytically and numerically the {\color{black} (spectral)} stability of travelling waves that are pervasive in the mathematical biological literature.
In this manuscript we illustrate this technique on two well-known models, both considered on an unbounded domain.}
The first is the Fisher/Kolmogorov--Petrovsky--Piscounov (F-KPP) equation 
\begin{equation}\label{eq:fish}
u_t = \epsi u_{xx} + u(1-u)\,, 
\end{equation}
and the second is a Keller--Segel (K-S) model of bacterial chemotaxis
\begin{equation} \label{eq:ks-pde-full}
\begin{split}
u_t & =  \ve u_{xx}- \alpha w\,,  \quad   \\
w_t & = \delta w_{xx} - \beta  \left( \frac{w u_x}{u} \right)_x\ .
\end{split}
\end{equation}

The F-KPP equation was chosen because in this case, the travelling wave stability analysis becomes {\em analytically} tractable (see \Cref{sec:appendix}). Further, the literature on stability of travelling waves in \cref{eq:fish} is vast (see for example \cite{murray03,sand02, sherrattetal14, saarloos03} and the references therein for a partial list of the stability results). It thus provides a well-known backdrop against which to verify our spectral calculations. The K-S model in \cref{eq:ks-pde-full} was chosen in order to highlight how to extend our methods beyond scalar PDEs. It is also convenient because explicit solutions to \cref{eq:ks-pde-full} can be found when $\ve =0$ \cite{FC00}. We are thus able to omit a time-consuming step (numerically finding the solutions to the travelling wave ODEs) and focus on setting up and analysing the linearised spectral problem. Explicit solutions are not necessary for our methods to work (as the F-KPP example shows) and we discuss how the stability analysis of travelling waves can be adapted to the K-S problem when $\ve \neq 0$ in \Cref{sec:smalldisp}. 

The F-KPP equation was first introduced by Luther in 1906 who originally used it to model and study travelling waves in chemical reactions \cite{luth1906}. It was named for Fisher, and for Kolmogorov, Petrovsky and Piscounov, who independently wrote seminal papers on the equation, using it to model the spread of a gene through a population \cite{fish37,kpp37}. In this manuscript, we assume the diffusion coefficient $\delta $ is strictly positive. 

\Cref{eq:ks-pde-full} was proposed in the 1970's by Keller and Segel \cite{KS71a, KS71b} to describe chemo-tactically-driven cell migration in which a population of bacteria exhibits an advective flux in response to a gradient of a diffusible secondary species (i.e. nutrient); see \cite{HP09, TMPA08} and references therein for a current overview of PDE models with chemotaxis. In \cref{eq:ks-pde-full}, the bacteria population density is denoted by $w(x,t)$ and the nutrient concentration by $u(x,t)$.  The model exhibits so-called logarithmic sensitivity and we assume a constant consumption rate function.
The diffusion of the nutrient is assumed to be much smaller than the diffusion of the bacteria population: $0 \leq \ve \ll \delta$. Finally, $\alpha>0$ models the rate at which nutrients are consumed, while $\beta>0$ measures the strength of the chemotaxis term. We also assume $0< \delta < \beta$. 

By a {\em travelling wave} solution of eqs. \eqref{eq:fish} or \eqref{eq:ks-pde-full}, we mean a solution to \cref{eq:fish} of the form $u(x-ct)$, or a pair of solutions to \cref{eq:ks-pde-full} of the form $(u(x-ct), w(x-ct))$ travelling from left to right with some positive (constant) wave propagation speed $c$.

To study travelling wave solutions, we introduce a moving coordinate frame; setting $z:= x-ct$, and $\tau := t$, \cref{eq:fish} becomes
\begin{equation}\label{eq:travellingpde} u_\tau = \epsi u_{zz} + cu_z +u(1-u)\,, \end{equation} 
while \cref{eq:ks-pde-full} becomes
\begin{equation}\label{eq:ks-travellingpde}
\begin{split}
u_\tau & =  \ve u_{zz}- \alpha w + cu_z\ , \\
w_\tau & = \delta w_{zz} - \beta  \left( \frac{w u_z}{u} \right)_z +cw_z\,.
\end{split}
\end{equation}

Travelling waves $u = \hat{u}(z)$, or $(u,w) = (\hat{u}(z),\hat{w}(z))$ will then satisfy the ODEs:
\begin{equation}\label{eq:travellingode}
\epsi u_{zz} + c u_z + u(1-u) = 0, 
\end{equation}
or
\begin{equation}\label{eq:ks-travellingode}
\begin{split}
\ve u_{zz}- \alpha w + cu_z\  & = 0, \\
 \delta w_{zz} - \beta  \left( \frac{w u_z}{u} \right)_z +cw_z & = 0\,.
\end{split}
\end{equation}

Once travelling waves to \cref{eq:fish,eq:ks-pde-full} have been found, we are next concerned with their stability. In particular, we are interested in the spectral stability of  the travelling waves. A full analysis of stability of travelling waves in the F-KPP and K-S equations is well beyond the scope of this manuscript. However, in the F-KPP equation, it is known that travelling waves of speed $c\geq 2\sqrt{\delta}$ are spectrally and linearly stable relative to certain perturbations (or in certain weighted spaces) and that the travelling waves of speed $0<c<2\sqrt{\delta}$ are unstable to all perturbations. One can find these results in a variety of sources relating to travelling waves, see for example \cite{murray03,sand02, sherrattetal14, saarloos03} and the references therein. For completeness, we include a proof in \Cref{sec:fkpp}. These results imply that there will not be any eigenvalues (with eigenfunctions in an appropriate space) with positive real part of the linear operator found by linearising \cref{eq:travellingpde} about travelling wave solutions. This is numerically confirmed by our calculations. 

Stability theory for travelling waves in chemotactic models is of course newer, and the full stability analysis does not appear to be known, but partial stability results for the K-S model we are considering in \cref{eq:ks-pde-full}, {\color {black} can be found in}, for example, \cite{nagai1991travelling,rosen1975stability}, and  the review paper by Wang \cite{wang2013}. In \Cref{sec:keller-segel} we focus on the model when $\ve = 0$. For the explicit travelling waves in this model, it is known that the continuous spectrum 
{\color{black} has} a nonzero intersection with the right half plane, and that such spectrum cannot be entirely weighted away \cite{nagai1991travelling}. This suggests the presence of so-called {\em absolute} spectrum in the right half plane. We numerically confirm this. We do not speculate on the effect that the absolute spectrum with positive real part may have on the dynamics (either long term or otherwise). It is also known that, for the travelling waves we consider in \Cref{sec:keller-segel}, the linearised operator does not have any real positive eigenvalues \cite{rosen1975stability}. There does not appear to be any proof in the literature of the absence of eigenvalues in the right half plane with nonzero imaginary part. We verify this fact for a large domain in the right half of the complex plane.  We further show that $\lambda =0$ is an eigenvalue of the linearised system, with multiplicity two. 

We illustrate our methods first for travelling waves in the F-KPP equation. In \Cref{sec:fkpp}, we set up the spectral problem, and find the continuous spectrum. We then define the point spectrum and Evans functions as a means to find it. We then introduce the Riccati equation and use it to construct a new, well-behaved function whose roots correspond to eigenvalues (a defining feature of an Evans function). We extend this function into the continuous spectrum in the natural way and show that its roots still correspond to eigenvalues of the linear operator.  
We conclude \Cref{sec:fkpp} with what is, to the best of our knowledge, a new proof of the lack of eigenvalues with positive real part of the linearised operator about a travelling wave in the F-KPP equation. In \Cref{sec:keller-segel}, we follow the same recipe and show how to apply the ideas of \Cref{sec:fkpp} to {\color{black} travelling waves in} systems of PDEs with more than one dependent variable, using as our example, {\color{black} a K-S model of bacterial chemotaxis}
\cref{eq:ks-pde-full} when $\ve = 0$. We set up the associated spectral problem, and compute the continuous spectrum, define eigenvalues, and the Evans function. We then compute the Riccati equation and accompanying machinery for this example and use it to define a new, well-behaved Evans function which can be readily extended into the continuous spectrum. We then use this function to numerically verify the absence of eigenvalues with positive real part in a large domain of the complex spectral plane. We conclude the section by numerically establishing that $\lambda=0$ is an eigenvalue of the linearised operator of multiplicity two. 
{\color{black} In \Cref{sec:summary}, we summarise our results and provide concluding remarks.}

\subsection{Acknowledgements} {\color{black} RM, GJP and MW} gratefully acknowledge the partial support of Australian Research Council grant ARC DP110102775. {\color {black} PvH gratefully acknowledges support under the Australian Research Council's Discovery Early Career Researcher Award funding scheme DE140100741.  KH also gratefully acknowledges support from an Australian Mathematical Society Lift-off Fellowship.}

\section{Travelling waves in the F-KPP equation} \label{sec:fkpp}

{\color{black}
We use a dynamical systems approach to analyse the travelling wave problem of the F-KPP equation, i.e.\,we write \cref{eq:travellingode} as a system of first-order equations:} 
\begin{equation}\label{eq:odesys} 
\frac{du}{dz}  = v\,, \quad \frac{dv}{dz} =  \frac{1}{\epsi}\left(-cv -u(1-u) \right)\,.\end{equation}
There are two equilibria of \cref{eq:odesys} in the $uv$-plane: one at $(0,0)$ and one at $(1,0)$. The Jacobian of \eqref{eq:odesys} is
$$ Df(u,v) = \begin{pmatrix} 0 & 1 \\ \frac{-1+2u}{\epsi} & \frac{-c}{\epsi}  \end{pmatrix}\,.$$
At the point $(1,0)$, the eigenvalues of $Df(u,v)$ are 
\begin{equation}
\mu_1^u = \frac{-c + \sqrt{c^2 + 4\epsi}}{2 \epsi}\,,  \quad \mu_1^s = \frac{-c - \sqrt{c^2 + 4\epsi}}{2 \epsi}\,.
\end{equation}
For all values of $c$ there is one positive eigenvalue and one negative eigenvalue. Thus, $(1,0)$ is a saddle point in the phase plane. At the point $(0,0)$, the eigenvalues of $Df(u,v)$ are
\begin{equation}
\mu_0^s = \frac{-c + \sqrt{c^2 - 4\epsi}}{2 \epsi} \,,  \quad \mu_0^{ss} = \frac{-c - \sqrt{c^2 - 4\epsi}}{2 \epsi}\,.
\end{equation}
For $c \geq 2\sqrt{\epsi}$ these are two real (distinct or equal), negative eigenvalues so $(0,0)$ is a (possibly degenerate) node. For $0< c <2\sqrt{\epsi}$, it is a stable focus. It is easy to see from the related phase portrait that for any value of $c > 0$ there is a heteroclinic orbit connecting $(1,0)$ to $(0,0)$. When $c^2\geq 4\epsi$ this orbit remains negative in {\color{black} $v$}, corresponding to a family of monotone travelling waves satisfying $\hat{u}(-\infty) = 1$ and $\hat{u}(+\infty) = 0$. When $c^2<4 \epsi$, there is a family of non-monotone travelling waves. See \Cref{fig:kpptravellingwave}.

\begin{figure}\label{fig:kpptravellingwave}
\includegraphics[scale=0.66]{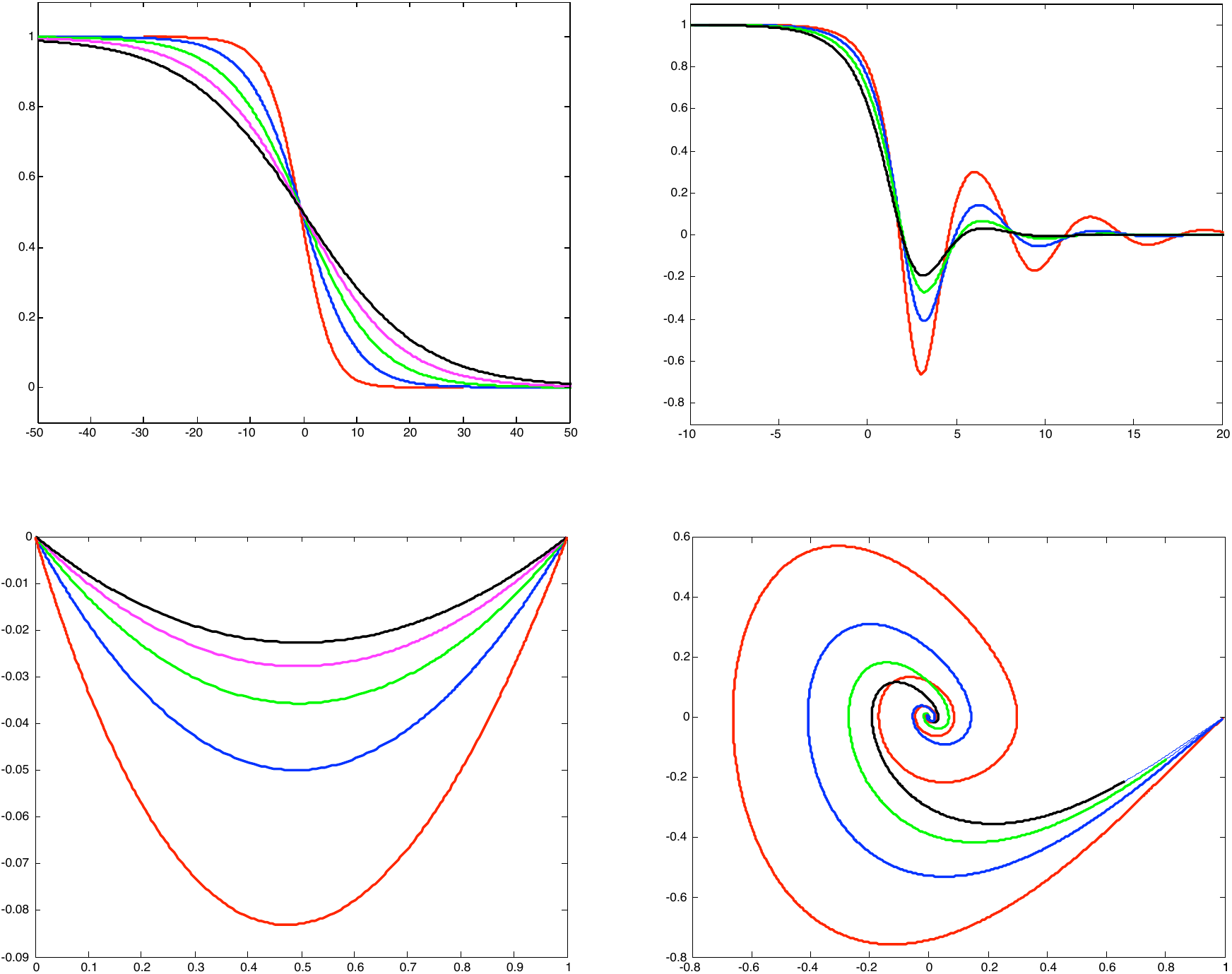}
{\color{black}\caption{Phase portraits and wave profiles for travelling waves in the F-KPP equation for various values of the wave speed $c$. In the figures on the left, we have $c\geq 2 \sqrt{ \epsi}$, (the plots show $c = 3, 5, 7, 9,$ and $11$, with $\epsi = 1$), while on the right, $0\leq c < 2 \sqrt{\epsi}$ ($c = 0.4,0.6,0.8,$ and $1$, with $\epsi = 1$). The top figures are the wave profiles while the bottom figures show that they indeed form heteroclinic connections in the phase portrait of \cref{eq:odesys} and that the origin is a stable node when $c\geq2\sqrt{\epsi}$ and a stable focus when $0 < c < 2 \sqrt{\epsi}$. } \label{fig:kpptravellingwave}}

\end{figure}

\subsection{{\color{black}The spectral problem}} \label{sec:fkppspec}

Travelling wave solutions to \eqref{eq:fish} are {\em steady state} solutions to \eqref{eq:travellingpde}. Once a travelling wave $\hat{u}(z)$ is found for a fixed $c$, we wish to consider how \cref{eq:travellingpde} behaves relative to perturbations (in the moving frame) about the travelling wave. We make the ansatz $u(z,\tau) = \hat{u}(z) + p(z,\tau)$, with $p(z,\tau)$ in an appropriate Banach space, substitute into \cref{eq:travellingpde} and consider only the first-order perturbative terms to give the formal (linearised) equation for $p$: 
\begin{equation}\label{eq:linear}
p_\tau = \epsi p_{zz} + cp_z + (1- 2\hat{u})p\,.
\end{equation}

Let $\cH^1(\R)$ denote the usual Sobolev space of functions from $\R$ to $\R$ which are square integrable and with first (weak) derivative also being square integrable. We define the (linear) operator $\cL: \cH^1(\R)  \to \cH^1(\R)$ by $$ \cL p := \left(\epsi \partial_{zz} + c \partial_z + (1-2\hat{u})\right) p\,.$$
Letting $\I$ be the identity map on $\cH^1(\R)$, we have the following definition: 
\theoremstyle{definition}
\begin{definition} We say that a $\lambda \in \C$ is in the {\em spectrum} of the operator $\cL$ if the operator $\cL - \lambda \I$ is not invertible on (some dense subset of) $\cH^1(\R)$. The set of all such $\lambda \in \C$ will be denoted as $\sigma(\cL)$.
\end{definition}

The operator $\cL - \lambda \I$ on $\cH^1(\R)$ is equivalent to the operator $\cT(\lambda) : \cH^1(\R) \times L^2(\R) \to \cH^1(\R) \times L^2(\R)$ given by 
\begin{equation}\label{eq:operator}
\cT(\lambda) \begin{pmatrix} p \\ q \end{pmatrix}:= \left( \dfrac{\rd}{\rd z} - A(z;\lambda) \right) \begin{pmatrix} p \\ q \end{pmatrix} := 
\begin{pmatrix} p \\ q \end{pmatrix}' - \begin{pmatrix} 0 & 1 \\ \frac{\lambda - 1 + 2 \hat{u}}{\epsi} & \frac{-c}{\epsi} \end{pmatrix} \begin{pmatrix} p \\ q  \end{pmatrix}. 
\end{equation}
Here we have defined $' := \dfrac{\rd}{\rd z}$ and the matrix $A(z;\lambda)$, and we further define
$$
A_\pm(\lambda) : = \lim_{z\to\pm \infty}A(z;\lambda) = \begin{pmatrix} 0 & 1 \\ \frac{\lambda \mp 1 }{\epsi} & \frac{-c}{\epsi} \end{pmatrix}\,.
$$ 
The (spatial) eigenvalues of $A_+(\lambda)$ are
\begin{equation}
\mu_+^u(\lambda) := \frac{-c + \sqrt{c^2 + 4\epsi(\lambda - 1)}}{2 \epsi} \,, \quad \mu_+^s(\lambda) := \frac{-c - \sqrt{c^2 + 4\epsi(\lambda-1)}}{2 \epsi}\,,
\end{equation}
and those of $A_-(\lambda)$ are 
\begin{equation}
\mu_-^u(\lambda) := \frac{-c + \sqrt{c^2 + 4\epsi(\lambda + 1)}}{2 \epsi} \,, \qquad \mu_-^s(\lambda) := \frac{-c - \sqrt{c^2 + 4\epsi(\lambda+1)}}{2 \epsi}\,.
\end{equation}
As it will often be convenient, when there is no ambiguity we will drop the arguments in the eigenvalues of the matrices $A_\pm(\lambda)$, writing instead $\mu_\pm^{s,u}$ as appropriate. We remark also that in the case that $\lambda = 0$, we have that $\mu^{s,u}_-(0) = \mu^{s,u}_1$, and that $\mu^{s,u}_+(0) = \mu^{ss,s}_0$ from before. Further, we have that the (spatial) eigenvectors of $A_\pm(\lambda)$ are 
$( 1, \mu_\pm^{u,s})^\top$. Lastly, we denote the subspaces spanned by the various eigenvectors of $A_\pm(\lambda)$ as $\xi^{u,s}_\pm (\lambda)$, respectively, again with the possibility of dropping the argument when convenient.  

\subsection{The Continuous Spectrum} \label{sec:fkppcont}We claim that the spectrum of the operator $\cL$ naturally falls into two parts: the continuous spectrum and the point spectrum. {\color{black} The point spectrum will be values $\lambda \in \sigma(\cL)$ such that $(\cL-\lambda)$ is a Fredholm operator of index zero. The continuous spectrum will be the complement of the point spectrum (in $\sigma(\cL)$).}
For the description of the continuous spectrum we follow \cite{sand02,sandscheel00}, while in order to best describe the point spectrum of $\cL$ we follow \cite{agj90, jones84, pegowein92}. There is no discrepancy with our choices, however, and equivalent statements for the point and continuous spectrum are found in all of \cite{jones84, kapitula2013spectral, pegowein92, sand02, sandscheel00}. We refer the reader to \cite{kapitula2013spectral, sand02} for a rigorous proof of the equivalence of all such definitions as well as the fact that our definition of the spectrum can indeed be broken up into the sets defined as the point and continuous spectrum, as given below. 

\begin{definition}\label{def:sig}
We recall that the {\em signature} of a matrix $M$, is the triple $(n_1, n_2, n_3)$ where the $n_j$'s are the dimensions of the positive, negative and null space of $M$ respectively. The signature will either be denoted by a triple of integers (e.g, $(2,1,0)$), or by an explicit list of the signs of the eigenvalues (e.g. $(+,+,-)$). See \Cref{fig:keller-segel-cont-spec}. 
\end{definition}

\begin{definition} \label{def:essential}
We define the {\em continuous spectrum} of the operator $\cL$, denoted $\sige(\cL)$ or sometimes just $\sige$, to be the set (in $\C$) of those $\lambda$ {\color{black} for which the signatures of $A_+(\lambda)$ and $A_-(\lambda)$ are not equal.}
\end{definition}
\noindent We note that one can track the real part of the eigenvalues of $A_\pm(\lambda)$ and that only one of the signs of $\re{\mu^u_\pm}$ will change {\color{black} as $\lambda$ is varied}. Further, in order for the sign of $\re{\mu^u_\pm(\lambda)}$ to change, there must be a $\lambda$ where $\re{\mu_\pm^u(\lambda)} =0$. Writing $\dfrac{\omega}{| \omega|}$ for the sign of a real number $\omega$, we have the following:
\begin{corollary}\label{cor:essential} The set $\sige(\cL)$ can be written as
$$
\sige{(\cL)} := \overline{ \left\{\lambda \in \C \quad  \bigg{|} \quad   \frac{\re{\mu_-^u }}{| \re{\mu_-^u}|}  \neq \frac{\re{\mu_+^u }}{| \re{\mu_+^u}|} \right\}}\,.
$$
\end{corollary}
The equations defining the boundary of the continuous spectrum are important in their own right and are the so-called {\em dispersion relations}.  {\color{black}These are where at least one of the eigenvalues of $A_+(\lambda) $ or $A_-(\lambda)$ is purely imaginary and are} given parametrically by
\begin{equation}\label{eq:dispersion}
\lambda = - \epsi k^2 \pm 1 + i c k \quad \textrm{for } k \in \R. 
\end{equation} 
Here $ik$ would be the imaginary eigenvalue of $A_\pm(\lambda)$. 
This describes two parabolas, opening leftward and intersecting the real axis at $\pm 1$. The complex plane minus the continuous spectrum is composed of two disjoint sets: $\C \setminus \sigma_{\textrm{c}} = \Omega_1\sqcup \Omega_2$. We define the sets in accordance with Definition \ref{def:essential} (see Figure \ref{fig:essential}):
\begin{eqnarray*}
\Omega_1 & := & \left\{ \lambda \in \C \setminus \sige |  0< \re{\mu_+^u(\lambda)} < \re{\mu_-^u(\lambda)} \right\}\,, \\
\Omega_2 & := & \left\{ \lambda \in \C \setminus \sige | \re{\mu_+^u(\lambda)} < \re{\mu_-^u(\lambda) }<0 \right\}\,. 
\end{eqnarray*}

\subsection{Eigenvalues} \label{sec:fkppeig} For a $\lambda \not \in \sigma_{\textrm{c}}$ we ask whether there are any nontrivial functions in the kernel of $\cT(\lambda)$. That is, can we find a nontrivial solution in $\cH^1(\R) \times L^2(\R)$ to the first order system 
\begin{equation}\label{eq:system}
\begin{pmatrix} p \\ q \end{pmatrix}' = \begin{pmatrix} 0 & 1 \\ \frac{\lambda - 1 + 2 \hat{u}}{\epsi} & \frac{-c}{\epsi} \end{pmatrix} \begin{pmatrix} p \\ q  \end{pmatrix} = A(z;\lambda) \begin{pmatrix} p \\ q \end{pmatrix}\, ?
\end{equation}
Any such solution must decay to $0$ as $z \to \pm \infty$ {\color{black} and as the next proposition illustrates, there is only one way that this can be realised.} 

\begin{figure}[t]
\centering
\includegraphics[scale=1]{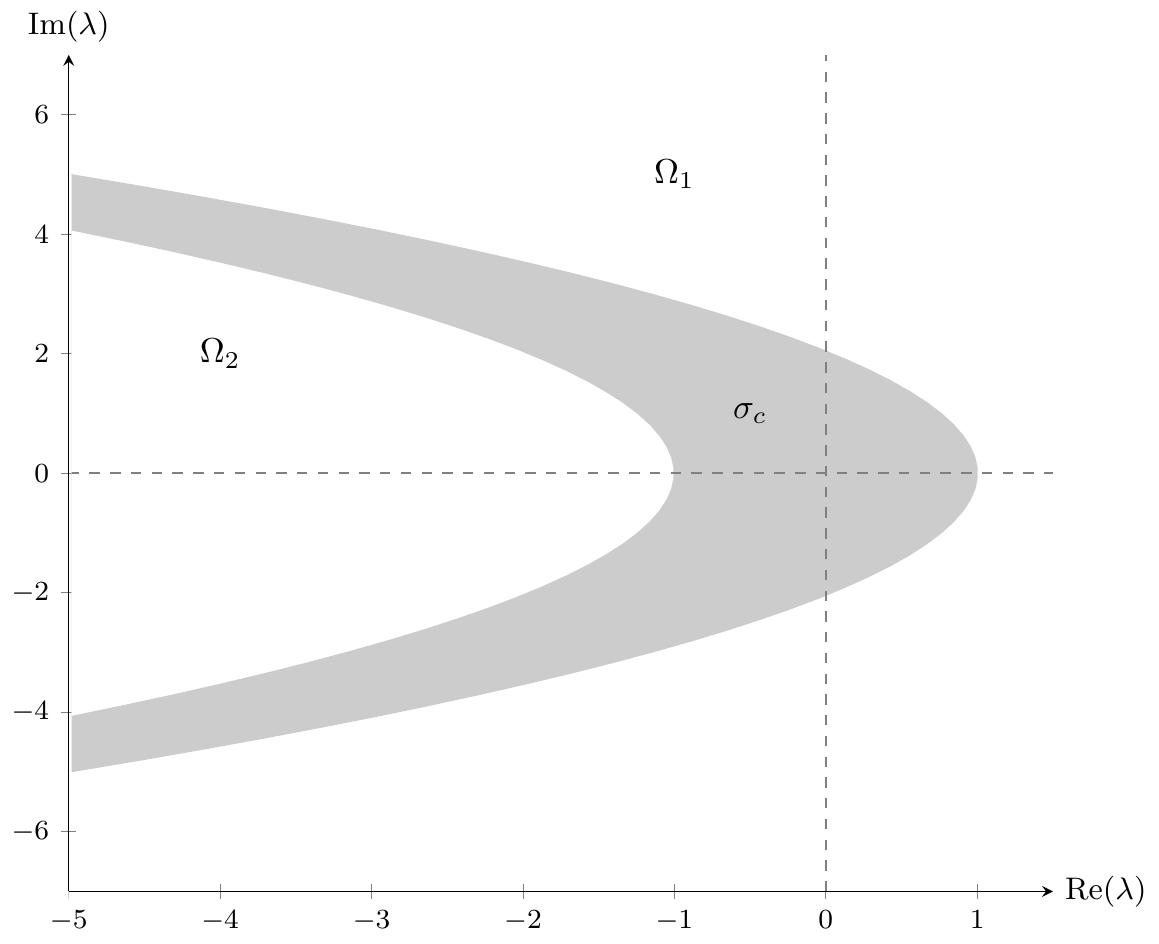}
\caption{The continuous spectrum of the linearised operator around a travelling wave (c.f. \cref{eq:operator} and \Cref{def:essential}). The boundary is given by \cref{eq:dispersion}. For the above picture we chose $\delta =1$ and $c = \dfrac{5}{\sqrt{6}}$. In $\Omega_1$ we have $0< \re{\mu_+^u} < \re{\mu_-^u}$, in $\Omega_2$ we have $\re{\mu_+^u} < \re{\mu_-^u}<0$ and in $\sigma_{\textrm{c}}$ we have $\re{\mu_+^u}<0 < \re{\mu_-^u}$.}
\label{fig:essential}
\end{figure}

\begin{proposition}\label{prop:geometric} For $\lambda \in \C \setminus \sige$, if $( p, q )$ is a solution to \eqref{eq:system} such that $( p, q ) \in \cH^1(\R) \times L^2(\R)$ then 
\begin{equation}
\lim_{z \to -\infty} \begin{pmatrix} p \\ q \end{pmatrix} \to \xi^u_- \quad \text{and} \quad
\lim_{z \to \infty} \begin{pmatrix} p \\ q \end{pmatrix} \to \xi^s_+ \, .
\end{equation}
That is, $( p, q )$ decays to the stable subspace $\xi^s_+$ of $A_+(\lambda)$ as $z \to + \infty$ and the unstable subspace $\xi_-^u$ of $A_-(\lambda)$ as $z \to -\infty$. 
\end{proposition}

A rigorous proof of this proposition can be found in \cite{jones84, kapitula2013spectral, pegowein92}. An intuitive reasoning behind why the proposition should be true is the following: For a $\lambda \in \Omega_1$, as $z \to \infty$, the system \eqref{eq:system} behaves like 
$$ 
\begin{pmatrix} p \\ q \end{pmatrix}' = A_+(\lambda) \begin{pmatrix} p \\ q \end{pmatrix}
$$
and since we are in the region $\Omega_1$, we have only one stable direction. Thus, if $(p, q) \to 0$, it must do so along the direction of the stable subspace $\xi^s_+$. The same is true as $z \to - \infty$ because we must have that the solution decays to zero along the subspace $\xi^u_-$. This argument also shows that for any $\lambda \in \Omega_2$ no solutions decay to zero as $z \to - \infty$. 
\begin{definition}\label{def:eigenvalue}
We will say that $\lambda \in \Omega_1$ is a (temporal) {\em eigenvalue} with {\em eigenfunction} $p$ if we can find such a solution to \eqref{eq:system} which is in $\cH^1(\R) \times L^2(\R)$. 
\end{definition}

We next exploit the linearity of \eqref{eq:system}. For a fixed $\lambda \in \Omega_1$ and for each $z \in \R$, let $\Xi^u(z;\lambda)$ be the linear subspace of solutions which decay to $\xi^u_-$ as $z \to -\infty$ and let $\Xi^s(z;\lambda)$ be the linear subspace of solutions which decay to $\xi^s_+$ as $z \to +\infty$. We note that in our example we can (for any fixed $\lambda$) view $\Xi^u(z;\lambda)$ and $\Xi^s(z;\lambda)$ as (line) bundles over $\R$. This justifies calling $\Xi^u(z;\lambda)$ `the unstable manifold' and $\Xi^s(z;\lambda)$ `the stable manifold'. What we mean by this is that $\Xi^u(z;\lambda)$ is the manifold of solutions that decay as $z \to - \infty$ to the unstable subspace of $A_-(\lambda)$ (and similarly for $\Xi^s(z;\lambda)$). We can evaluate $\Xi^u$ and $\Xi^s$ at a fixed value $z_0$ and if they are linearly dependent then we will have an eigenvalue. This is because of uniqueness of solutions to ODEs; if they agree at one $z_0$ then they must agree for all $z \in \R$ and so we have (a linear subspace) of solutions which decay as $z \to \pm \infty$.

Let \begin{equation*}w^u(z;\lambda) = \begin{pmatrix} w_1^u(z;\lambda) \\ w^u_2(z;\lambda) \end{pmatrix} \quad\text{and}\quad w^s(z;\lambda) =  \begin{pmatrix} w_1^s(z;\lambda) \\ w^s_2(z;\lambda) \end{pmatrix}\end{equation*} be two solutions in $\Xi^u$ and $\Xi^s$, respectively. These are two vectors in $\C^2$ and we know that $\lambda$ is an eigenvalue if and only if they are linearly dependent for some (and hence every) $z_0 \in \R$. For convenience, we choose $z_0 = 0$. 
We have therefore shown the following: 
\begin{proposition} \label{prop:defevans}The complex number $\lambda \in \Omega_1$ is an eigenvalue if and only if 
\begin{equation}
D(\lambda) := \det{ \begin{pmatrix}  w_1^u(0;\lambda) & w_1^s(0;\lambda) \\ w_2^u(0;\lambda) & w_2^s(0;\lambda) \end{pmatrix} } = 0. 
\end{equation}
\end{proposition}
\begin{definition} \label{def:evans}
The function $D(\lambda)$ defined in Proposition \ref{prop:defevans} is called {\em an  Evans function}. 
\end{definition}

\subsection{The Riccati equation} \label{sec:riccati}
For Definition \ref{def:evans}, we only compare two possible appropriately decaying solutions to the ODE \eqref{eq:system}.
In the following, we are interested in whether or not a pair of {\em subspaces} intersect, rather than the particulars of any given solution. 
\begin{definition}The set of (complex) one-dimensional subspaces in complex two-space is called {\em complex projective space} and is denoted $\C P^1$.
\end{definition}
Complex projective space can be given the structure of a complex manifold of one complex dimension and is topologically equivalent to the Riemann sphere, which we denote by $S^2$. A line in $\C^2$ through the origin is determined by a pair of complex numbers denoted $[p:q]$ that are not both zero. We can write down all the lines where $p \neq 0$ as $[1: \eta ]$ and we see right away that this is (equivalent to) a copy of the complex plane. Similarly, we write all the lines where $ q \neq 0$ as $[\tau:1]$ and so this too is equivalent to a copy of the complex plane.
Further, for any line except for two (where $p$ or $q = 0$), we have that $\eta = \dfrac{1}{\tau}.$ These are the typical {\em charts} on $\C P^1$. {\color{black} For a given two-dimensional system of linear first-order ODEs, we get an equivalent (nonlinear, non-autonomous) flow on $\C P^1$: the so-called {\em Riccati equation}}. 

We obtain an expression (on each chart) for the Riccati equation by simply differentiating the defining relations of $\eta$ and $\tau$ and using \cref{eq:system}.  We get
\begin{equation}\label{eq:riccati}
\begin{split}
\eta' &= \left( \frac{q}{p} \right)' =  \frac{1}{\epsi}\left(\lambda - 1 + 2 \hat{u} \right) - \frac{c}{\epsi}\eta - \eta^2\,, \\ 
\tau' &=  \left( \frac{1}{\eta} \right)' = \frac{-\eta'}{\eta^2} =   1 +  \frac{c}{\epsi}\tau -  \frac{1}{\epsi}\left(\lambda - 1 + 2 \hat{u} \right)\tau^2\,,
\end{split}
\end{equation}
being two first-order non-autonomous nonlinear ODEs. Further, we have that $\xi^u_\pm$ and $\xi^s_\pm$ will be {\em fixed points} of these systems. To see this in coordinates, we have that in the $\eta$ chart, $\xi^{u,s}_\pm$ is given by the eigenvalues $\mu^{u,s}_\pm$, while in the $\tau $ chart, they are the multiplicative inverses, a feature of {\color{black} \cref{eq:system} } that will not in general be true for an arbitrary {\color{black} two-dimensional system of first order ODEs.} We also have that $\lambda$ will be an eigenvalue if and only if we can find a heteroclinic connection between $\xi^u_-$ and $\xi^s_+$. In terms of the $\eta$ chart, this is a heteroclinic connection between $\mu^u_-$ and $\mu^s_+$ (and between their multiplicative inverses in the $\tau$ chart). 

\begin{remark} By writing out the real and imaginary parts of the flow in the $\eta$ and $\tau$ charts and by considering the flow direction on the real axis, one can show that there cannot be a heteroclinic connection in the case of a spectral parameter with non-zero imaginary part. Further, by applying techniques used in \cite{rmckrtj12} one can similarly show that there are no real, positive eigenvalues. In Section \ref{sec:appendix}, we exploit this idea to prove the absence of eigenvalues in the case of travelling waves in the F-KPP equation.
\end{remark}

We determine a related Evans function by letting $\eta^u(z;\lambda)$ and $\eta^s(z;\lambda)$ be the solutions which decay to $\xi^u_-$ and $\xi^s_+$, respectively, in the $\eta$ chart. Moreover, suppose that $\eta^{s,u}(z;\lambda)$ is finite for all $z \in \R$. This corresponds to $w_1^{s,u}(z;\lambda)$ being non-zero or staying in a single chart. This requirement is not necessary and we discuss what happens {\color{black} (see \cref{subsec:charts})} if we need to leave the chart, below, but we include it here for convenience. We define a new function 
\begin{equation}\label{eq:defeeta}
\begin{split}
E_\eta(\lambda) := &  \frac{D(\lambda)}{w^u_1(0;\lambda)w^s_1(0;\lambda)} \\
  = &  \frac{w_1^u(0;\lambda)w_2^s(0;\lambda) - w_2^u(0;\lambda)w_1^s(0;\lambda)}{w^u_1(0;\lambda)w^s_1(0;\lambda)}\\ 
 = & \eta^s(0;\lambda) - \eta^u(0;\lambda)\,.
\end{split}
\end{equation}
We define the functions $\tau^u(z;\lambda)$ and $\tau^s(z;\lambda)$ and the corresponding Evans function in the $\tau$ chart similarly. Here, though, we have 
\begin{equation}
\begin{split}
E_\tau(\lambda) := &  \frac{D(\lambda)}{w^u_2(0;\lambda)w^s_2(0;\lambda)} \\
  = &  \frac{w_1^u(0;\lambda)w_2^s(0;\lambda) - w_2^u(0;\lambda)w_1^s(0;\lambda)}{w^u_2(0;\lambda)w^s_2(0;\lambda)}\\ 
 = & \tau^u(0;\lambda) - \tau^s(0;\lambda)\,.
\end{split}
\end{equation}
Note that for $\lambda \in \Omega_1$, the function $E_{\eta}(\lambda)$ is zero if and only if $\eta^s(0;\lambda) = \eta^u(0;\lambda)$. By uniqueness of solutions to ODEs, we therefore have that $\eta^u(z;\lambda) = \eta^s(z;\lambda)$ for all $z \in \R$ and, hence, a heteroclinic connection between $\xi^s_+$ and $\xi^u_-$ exists. This will be true if and only if $\lambda \in \Omega_1$ is an eigenvalue. The same argument holds for $E_\tau(\lambda)$, i.e., we will have an eigenvalue if and only if $E_\tau(\lambda) =0$. 

We only need to calculate the $\eta$'s to compute the Evans functions. Since $\tau = \dfrac1\eta$, we have 
\begin{equation}
\begin{split}
E_\tau(\lambda) = &  \tau^u(0;\lambda) - \tau^s(0;\lambda) = \frac{1}{\eta^u(0;\lambda)} - \frac{1}{\eta^s(0;\lambda) }\\ 
= & \frac{\eta^s(0;\lambda) - \eta^u(0;\lambda)}{\eta^u(0;\lambda)\eta^s(0;\lambda)} \\
=&  \frac{E_\eta(\lambda)}{\eta^u(0;\lambda)\eta^s(0;\lambda)}\,,
\end{split}
\end{equation}
so knowing how to compute $E_\eta(\lambda)$ is enough to compute $E_\tau(\lambda)$. 

We are interested in where the function $D(\lambda)$, and hence $E_\eta(\lambda)$, is equal to zero in the region $\Omega_1$, assuming that $w_1^{u,s}(0;\lambda) \neq 0$ for any $\lambda \in \C$. To investigate this we exploit the analyticity (or continuity) of $E_\eta(\lambda)$ for $\lambda \in \Omega_1$. We appeal to a theorem from complex analysis (see, for example, \cite{conway73}), which says that if $f(\lambda)$ is a meromorphic function on some simply connected domain $\Omega \subseteq \C$ with no zeros or poles on a {\color{black} closed} curve $\gamma(t) \subseteq \Omega$ in the complex plane, then, letting $N$ denote the number of zeros of $f$ inside $\gamma(t)$ and $P$ 
denote the number of poles inside $\gamma(t)$, we have
\begin{equation}
N - P = \frac{1}{2\pi i} \oint_{\gamma} \frac{\dot{f}(\lambda)}{f(\lambda)} d \lambda\,,
\end{equation}
where $\dot{}$ denotes $\dfrac{d}{d\lambda}$. 

The assumption that $w_1^{u,s}(0;\lambda) \neq 0$ for any $\lambda \in \Omega_1$ means that we stay in a single chart for each $\lambda \in \Omega_1$. Thus, we have that $E_{\eta} (\lambda)$ will not just be meromorphic but {\em analytic}, that is $P \equiv 0$, leading to the following:
\begin{proposition} Let $\gamma(t):[0,1]\to \C $ be a simple closed curve in the complex plane, oriented counterclockwise and let $D(\lambda)$, $E_\eta(\lambda)$ and $w^{s,u}_1(0;\lambda)$ be defined as above. Suppose that $w^{s,u}_1(0;(\gamma(t))) \neq 0$ for all $t\in[0,1]$. Then,
$$ 
\oint_{\gamma(t)} \frac{\dot{E_{\eta}}(\lambda)}{E_{\eta}(\lambda)} d \lambda = \oint_{\gamma(t)}\frac{\dot{D}(\lambda)}{D(\lambda)} d \lambda\,.
$$
That is, the number of zeros of $E_{\eta}(\lambda)$ is the same as the number of zeros of the Evans function. 
\end{proposition}

For every value of $\lambda \in \Omega_1$, the assumption that $w_1^{u,s}(0;\lambda) \neq0$ is consistent with the behaviour of numerical solutions to the F-KPP equation. Further, $N$ may be thought of as the {\em winding number} of $E_\eta(\lambda)$: the number of times $E_\eta(\lambda)$ (and hence $D(\lambda)$) winds around the origin (with a counter clockwise orientation), counted with sign, as we traverse $\gamma(t)$. Hence we can {\em visually} determine the number of zeros of $D(\lambda)$ in a closed contour in $\C$ in the case of F-KPP for quite large values of $\lambda$. 

\begin{figure}
\centering
\includegraphics[scale=0.93]{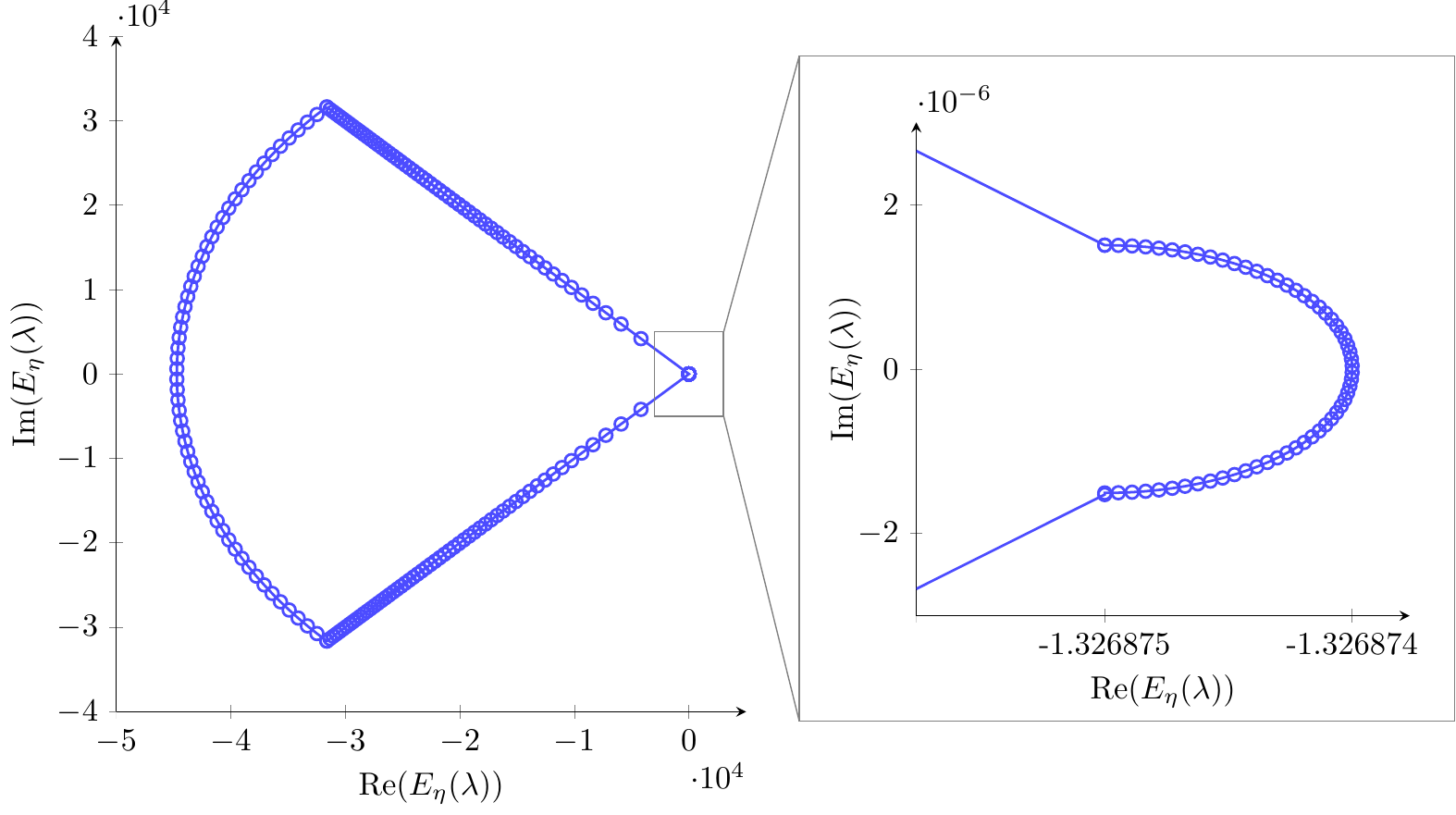}
\includegraphics[scale=0.93]{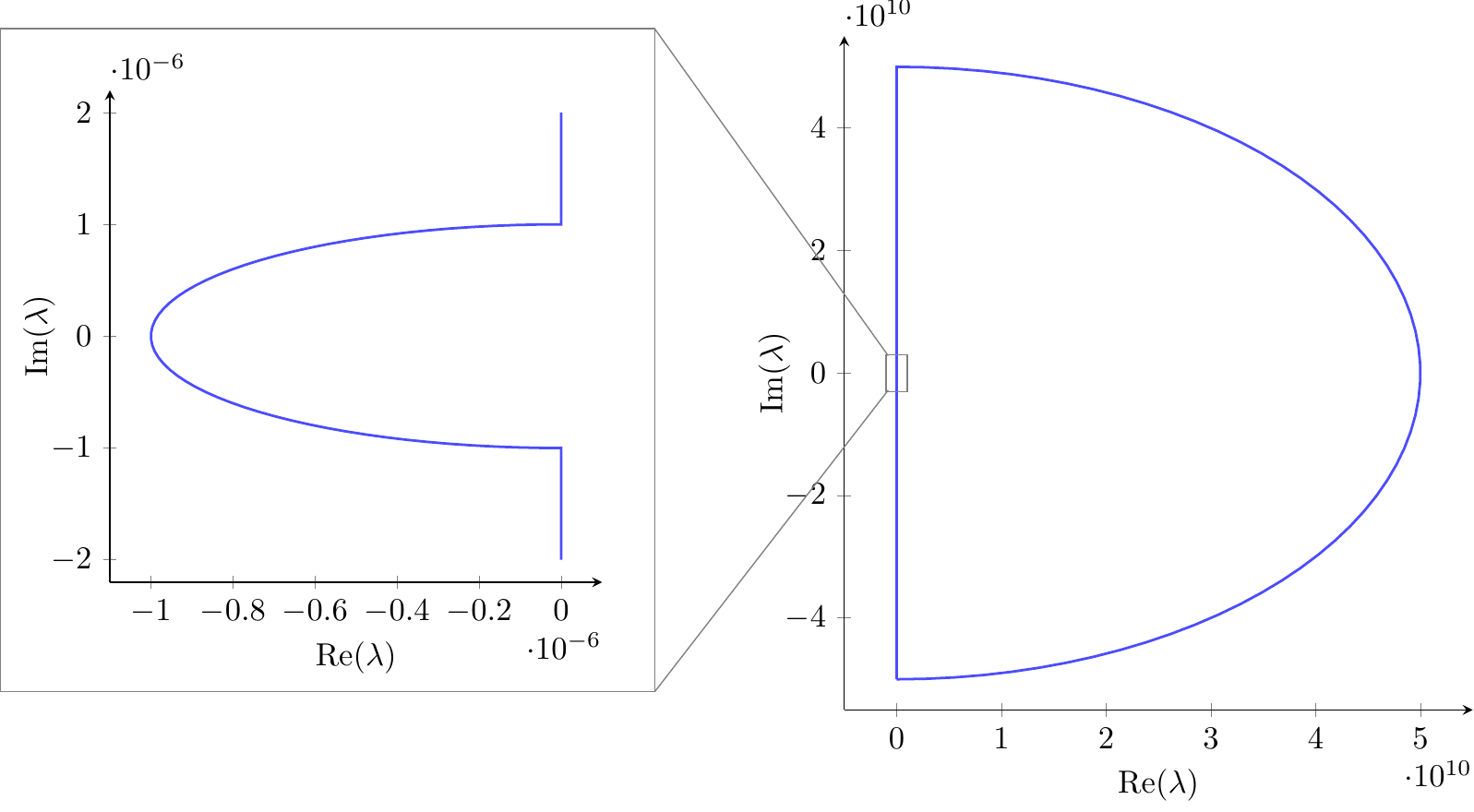}
\caption{Top: A plot of the function $E_\eta(\lambda)$, defined in \cref{eq:defeeta}, for the contour in the spectral parameter (bottom). The wave speed $c = 2.4$ and in these figures $\epsi = 1$. As can be clearly seen, $E_\eta(\lambda)$ does not wind around the origin, even though the curve of spectrum does. This confirms that $\lambda =0$ is not an eigenvalue in the sense of \Cref{def:eigess} and also numerically confirms the results in \Cref{sec:appendix}. Moreover, the spectral radius is quite large in this case ($5\times10^{10}$), while the function $E_\eta(\lambda)$ remains relatively well-behaved, even through fairly large changes of scale in the spectral parameter.}
\label{fig:fisherevanseta}
\end{figure}

As can be seen in Figure \ref{fig:fisherevanseta}, there are no eigenvalues in $\Omega_1$ in the right half plane with $|\lambda| \leq 5 \times 10^{10}$. It is evident that the winding number of $E_\eta(\lambda)$, and hence $D(\lambda)$, is zero. 

\subsection{Extending $E_\eta(\lambda)$  into the continuous spectrum}\label{sec:fkppextends} Since the goal of Evans function computations is to numerically infer stability or otherwise, we need to concern ourselves with values of the spectral parameter in the right half plane (that is, with $\re{\lambda} \geq 0$), not just those in $\Omega_1$. To this end, we need to consider values of $\lambda$ inside $\sigma_{\textrm{c}}$, the continuous spectrum, and re-visit our definition of (temporal) eigenvalues. We proceed in the manner outlined in \cite{jones84} and \cite{pegowein92}. 

Using Definition \ref{def:eigenvalue} for all $\lambda$ with $\re{\lambda} \geq 0$, if $\lambda \in \sigma_{\textrm{c}}$, then the matrix $A_-(\lambda)$ has two (spatial) eigenvalues, both with negative real parts. This implies that {\em every} solution of \cref{eq:system} decays to zero as $z \to +\infty$. In particular, any solution which decays to $0$ as $z \to -\infty$ will decay to zero as $z \to +\infty$, so, if we were to just require the existence of a solution decaying as $z \to \pm \infty$, we would see that {\em every} $\lambda \in \sige$ would be an eigenvalue. Moreover, it is straightforward to see that these solutions are indeed in $\cH^1(\R) \times L^2(\R)$.  

This would seem to suggest that {\em every} travelling wave is {\color{black} spectrally} unstable, and the linearised operator, linearised about every wave has eigenvalues with positive real part. This is at odds with with numerical experiments as well as known stability results: for example it has been known since its inception that the F-KPP wave of speed $c = 2 \sqrt{\delta}$ is stable relative to many compactly supported perturbations \cite{kpp37}, and moreover a wide variety of initial profiles will evolve in time to this wave (or at least, a closely related one) \cite{uchiyama1978}.  So in some sense we would like to say this wave is `stable' but we would also like to reconcile this notion with the idea that the linearised operator about a stable travelling wave should not have eigenvalues in the right half plane.
 {\color{black} We are thus motivated to make the following amendment to Definition \ref{def:eigenvalue}} 

\begin{definition}\label{def:eigess}
For a $\lambda \in \C$ with $\re{\lambda} \geq0$, we say that $\lambda$ is an {\em `eigenvalue'} if there is a solution to the Riccati equation that decays to $\xi_-^u$ as $z \to - \infty$ and to $\xi_+^s$ as $z \to + \infty$. 
\end{definition}

{\color{black}
\begin{remark} \label{rem:weighted}
We remark that the apparent contradiction which led to \Cref{def:eigess} can be resolved by the introduction of so-called {\em weighted spaces}. This amounts to restricting perturbations to those which decay faster than a given prescribed rate $\nu$, (in this example the space is denoted $\cH^1_\nu$). Subsequently the spectrum is shifted, and one chooses $\nu$ (if possible) so that the spectrum is shifted into the negative half plane. Thus there are no eigenvalues in the continuous spectrum with eigenfunctions in $\cH_\nu^1$. The wave is then said to be stable relative to these weighted perturbations (provided of course that there are no other eigenvalues with positive real part and with eigenfunctions in this weighted space).  

We claim that in the F-KPP travelling wave case, the presence of an `eigenvalue' corresponds {\em exactly} to weighted instability for all weight functions which shift the continuous spectrum into the left half plane. The right edge of the continuous spectrum is moved to the point $\lambda = 1+ c \nu+ \delta \nu^2$ in the weighted space $\cH_\nu^1$. This will be to the left of the right edge of the continuous spectrum in the unweighed space $\cH^1$ provided $\nu \in \left(\frac{-c - \sqrt{c^2-4\delta}}{2\delta}, \frac{-c+\sqrt{c^2-4\delta}}{2\delta} \right)$. Moreover, this will be in the left half plane only if $c^2<4\delta$.

%For $\lambda \in \sige$, the weights which move the continuous spectrum to the negative half plane are those in the range $(-\mu^u_+, -\mu^s_+)$. That is for any $\nu \in (-\mu^u_+, -\mu^s_+)$, the continuous spectrum in $\cH^1_\nu$ will be (entirely) in the left half plane. So perturbations in this weighted space must decay faster than $e^{-\mu^u_+ z}$ as $z \to + \infty$. 

Consider $\lambda \in \sige (\cL)$ (on the unweighted space $\cH^1$), with $\re{\lambda} >0$. If $\lambda$ is not an `eigenvalue', then all such solutions to \cref{eq:system} will decay exactly like $e^{-\mu^u_+ z}$ as $z \to +\infty$, and thus there can be no eigenfunctions in the weighted spaces $\cH_\nu^1$ for $\nu$ which shift the continuous spectrum to the left. 
However, if $\lambda$ is an `eigenvalue' then this indicates that there will be a solution with a decay rate faster than any weight function which will move the continuous spectrum to the left. Thus it will remain an `eigenvalue' for all weighted spaces with weights $\nu$ shifting the continuous spectrum to the left. We can therefore conclude that there is a point in the spectrum which will not be moved into the left half plane in any such weighted space. 
\end{remark}
}

%\newpage

\begin{remark}\label{rem:pego}
\Cref{def:eigess} follows the definition of `eigenvalue' from \cite{pegowein92}. The roots of $E_\eta(\lambda)$ and $E_\tau(\lambda)$ will detect the values $\lambda$ where we have a solution decaying with the {\em maximal} exponential rate as $z \to \pm \infty$. It is obvious that this definition agrees exactly with our definition of an eigenvalue in the region $\Omega_1$. Inside the continuous spectrum we will not allow our eigenfunction to decay to $0$ in just any fashion, it needs to decay along the (now strongly) stable subspace $\xi^s_+$. Since we will be primarily interested with the zeros of $E_{\eta,\tau}(\lambda)$ and given the discussion in \Cref{rem:weighted}, we drop the quotation marks, and simply refer to any such $\lambda$ as a (temporal) eigenvalue of the linearised operator $\cL$. 
\end{remark}

With Definition \ref{def:eigess}, eigenvalues still correspond exactly to zeros of $E_\eta(\lambda)$. Moreover, it is straightforward to see that we can still define $E_\eta(\lambda)$ to be analytic as we extend $\lambda$ into $\sige$. We can, in fact, use some analysis of the Riccati equation on the $\eta$ chart of $\C P^1$ to see when exactly we get a zero of $E_{\eta,\tau}(\lambda)$. To begin with, we fix a $\lambda \in \sige$ with $\re{\lambda} \geq 0$ and seek a heteroclinic connection between $\mu^u_-$ and $\mu^s_+$. For a general $\lambda$ we have that the unstable orbit in the $\eta$ chart coming from $\mu^u_-$ (viewed as a subspace in $\C^2$) will tend towards the steady state solution $ \mu^{sc}_+ : = \mu^u_+$ (here, because we are in $\sige$ we denote $\mu^{sc}_+ := \mu^{u}_+$ to note that it is in fact a stable fixed point of the Riccati flow on the $\eta$ chart). 

Recall that 
$$ 
\mu^{sc}_+ = \frac{-c + \sqrt{c^2 + 4\epsi(\lambda -1)}}{2 \epsi} \quad \textrm{and} \quad
\mu^{s}_+ = \frac{-c - \sqrt{c^2 + 4\epsi(\lambda -1)}}{2 \epsi}\,,
$$
which are different points in the chart of $\C P^1$ for all values of $\lambda$ except when $\displaystyle c^2+4\epsi(\lambda-1) = 0$, that is, except for $\lambda = 1 - c^2/(4\epsi)$. At this value of $\lambda$, we have  $\mu_+^s = \mu_-^{sc}$ and so what was a heteroclinic connection (of the Riccati flow on the $\eta$ chart of $\C P^1$) between the fixed points $\mu_-^u$ and $\mu_+^{sc}$, is also a heteroclinic connection between the fixed points $\mu_-^u$ and $\mu_+^{s}$.  Consequently, we have a zero of $E_\eta(\lambda)$ and this value of $\lambda$ will be an eigenvalue according to Definition \ref{def:eigess}. 

We observe that if $c^2 \geq 4 \epsi$ then the largest root, say $\tilde{\lambda}$, of the function $E_\eta(\lambda)$ is real and negative. However, as $c \searrow 2 \sqrt{\epsi}$, we have that $\tilde{\lambda}$ tends towards $0$ and if $c<2\sqrt{\epsi}$, then $E_\eta(\lambda)$ has a real, positive root. Thus, we have an eigenvalue $\tilde{\lambda}$ in the right half plane, which (evidently) destabilises the travelling wave. This corresponds with numerical experiments as well as the analytic results proven in \cite{hagan81}.
\begin{figure}
\includegraphics{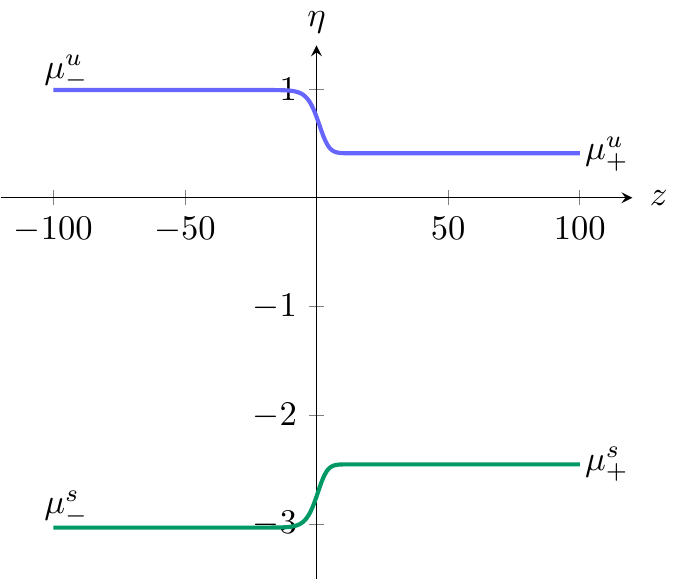} \quad
\includegraphics{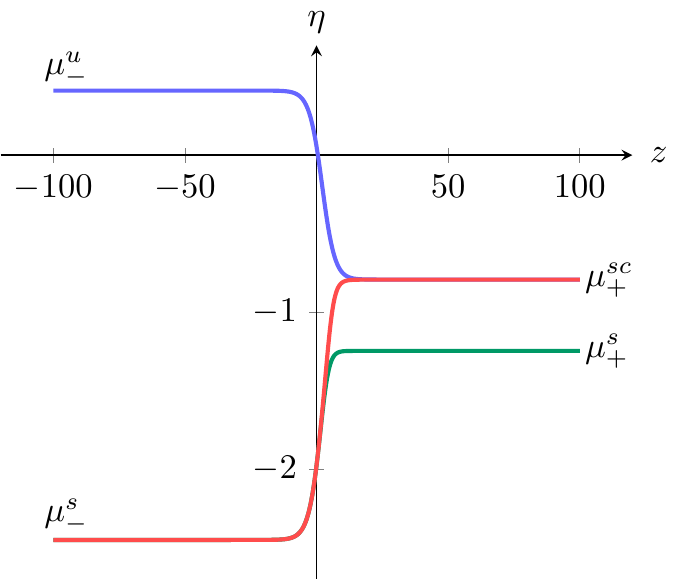}
\caption{Plots of the relevant solutions to the Riccati equation, \cref{eq:riccati}, in the $\eta$ chart for two (real) values of $\lambda$. On the left we have $ \lambda = 2$ and it is obvious that a heteroclinic connection does not exist. On the right, $\lambda = 0.01$ and we can see that the solution tending to $\mu_-^u$ as $z \to -\infty$ (the upper solution, blue online) tends towards $\mu_+^{sc} $ as $z \to + \infty$ (the middle solution, red online). Thus, while we have a traditional eigenvalue, we do not have an `eigenvalue' according to Definition \ref{def:eigess}.}
\label{fig:etachart}
\end{figure}

\subsection{Switching Charts} \label{subsec:charts}

Suppose that for some fixed $z_0$ we had that the solution of our Riccati equation $|\eta^u(z;\lambda)| \to \infty$, implying that the corresponding solution in the $\tau$ chart must tend to $0$. Given the uniqueness of solutions to ODEs on manifolds, we can find a value $z_1 < z_0$ such that $|\eta^u(z_1;\lambda)| < \infty$ and so consider the corresponding initial value problem in the $\tau$ chart where $\tau^*(z_1;\lambda) = \dfrac{1}{\eta^u(z_1;\lambda)}$. Evolving the $\tau$ problem from $z_1$ to a new $z_2 > z_0$ (noting along the way that $\tau^*(z_0;\lambda) = 0$), we can then consider the solution of the Cauchy problem on the $\eta$ chart with initial condition $\eta^u(z_2;\lambda) = \dfrac{1}{\tau^*(z_2;\lambda)}$. In this way, we have moved beyond the singularity of our Riccati solution. The impact that this strategy has upon our previously defined Evans functions needs to be explored. Given that we are no longer in the case where the number of poles of $E_\eta(\lambda)$ (the value $P$ above) is $0$, our winding number calculation becomes 
\begin{equation} \label{eq:mainthm}
\oint_{\gamma(t)} \frac{\dot{E_{\eta}}(\lambda)}{E_{\eta}(\lambda)} d \lambda = \oint_{\gamma(t)}\frac{\dot{D}(\lambda)}{D(\lambda)} d \lambda  - \oint_{\gamma(t)} \frac{\dot{w}^s_1(0;\lambda)}{w^s_1(0;\lambda)} d \lambda  - \oint_{\gamma(t)} \frac{\dot{w}^u_1(0;\lambda)}{w^u_1(0;\lambda)} d\lambda\,.
\end{equation}

We elaborate on the meaning of this result in the following theorem:
\begin{theorem}\label{th:cool} Let $\gamma(t)$ be a parametrised curve in the complex plane such that $D(\lambda)$ is analytic and has no zeros on $\gamma(t)$. Then, the winding number of $E_\eta$ along $\gamma(t)$ is the number of eigenvalues of $\cL$ inside that curve minus the number of poles of $E_\eta(\lambda)$ inside $\gamma(t)$.  
\end{theorem}

\begin{proof} Rearranging the definition of $E_\eta(\lambda)$ we have
$$ D(\lambda) = w^s_1(0;\lambda) w^u_1(0;\lambda) E_\eta(\lambda)\,.$$
Choosing a curve $\gamma(t)$ such that $D(\lambda)$ has no zeros on $\gamma(t)$ (that is, avoiding any eigenvalues) and such that $D(\lambda)$ is analytic on $\gamma(t)$, then applying the chain rule to logarithmic differentiation and rearranging, we have \cref{eq:mainthm}. $E_\eta$ has a pole exactly when either $w_1^u(0;\lambda)$ or $w_1^s(0;\lambda)$ is zero, and eigenvalues of $\cL$ are the zeros of $D(\lambda)$. 
\end{proof}

In the case of the F-KPP equation, this theorem enables us to find the number of eigenvalues inside any bounded contour, except those containing the so-called `absolute spectrum' where the function $E_\eta(\lambda)$ has a branching point of its domain. Following \cite{kapitula2013spectral}, we define the {\em absolute spectrum} as the $\lambda \in \C$ such that the real parts of $\mu_+^{u,s}$ or $\mu_-^{u,s}$ coincide. These can be determined as $\lambda \leq 1 - {c^2}/({4\epsi})$ in the case of $\mu_+$, and $\lambda \leq -1 - {c^2}/({4\epsi})$ for $\mu_-$ (note that $\lambda \in \R$ in the absolute spectrum). This offers another mechanism for destabilisation of the waves as $c \searrow 2 \sqrt{\epsi}$, namely that the absolute spectrum moves into the right half plane. In this case, the loss of meromorphicity of $E_\eta(\lambda)$ coincides exactly with the leading edge of the absolute spectrum. 

\subsection{A proof of the absence of eigenvalues with positive real part.}\label{sec:appendix}

The proof proceeds in two parts. For the first part, we show that there are no eigenvalues with non-zero imaginary part. For the second, we show the absence of a real positive eigenvalue when $c\geq2 \sqrt{\delta}$. 

\subsubsection{No complex eigenvalues}

Recalling \cref{eq:riccati}, we have that on the $\eta$ chart of $\C P^1$, the linearisation is given by 
\begin{equation}\label{eq:etariccati}
\eta'  = \frac{1}{\epsi}(\lambda + 1 - 2 \hat{u}) - \frac{c}{\epsi}\eta - \eta^2\,.
\end{equation}
Writing $\eta = \alpha + i \beta$ and $\lambda = m + i n$ with $m \neq 0$, we have that \cref{eq:etariccati} becomes (when viewed on $\C \approx \R^2$)
\begin{equation}
\begin{split}
\alpha' & = \frac{1}{\epsi}\left(1 - \alpha c + m - 2 \hat{u} \right) - \alpha^2 + \beta^2\,, \\
\beta'  & = - 2 \alpha \beta - \frac{\beta c}{\epsi} + \frac{n}{\epsi}\,.
\end{split}
\end{equation}
We see that on the line $\beta = 0$, we have that $\beta' = \dfrac{n}{\epsi}$ and so the sign of $\beta'$  
is the same as the sign of the imaginary part of the eigenvalue parameter $\lambda$.  Consequently, the flow is pointing towards the upper half plane when $\imag{\lambda} >0$ and towards the lower half plane when $\imag{\lambda}<0$. Now we have that an eigenvalue $\lambda$ on this chart is a value of $\lambda$ such that there is a connection under this flow from $\mu_-^u$ to $\mu_+^s$. 

Given the previous statement about the direction of the flow on the real axis of this chart, we claim that $\imag{\mu_-^u}>0$  and $\imag{\mu_+^s}<0$ if $\imag{\lambda} >0$, and the reverse inequalities if $\imag{\lambda}<0$. Thus, a connection is impossible, as long as $\imag{\lambda} \neq 0$. Proceeding directly we have that 
$$\mu_-^u = \frac{-c + \sqrt{c^2+4\epsi(\lambda +1)}}{2 \epsi}$$
and that 
$$ \imag{\mu_-^u} = \frac{\varkappa}{{\color{black}2}\epsi} \sin \left( \frac{1}{2}  \arg \left( c^2 + 4 \epsi(1+\lambda) \right) \right)\,,$$
where {\color{black} $\varkappa  := \sqrt{|c^2 + 4 \delta(\lambda+1)|} >0$} and by $\arg$, we mean the {\color{black} principal} argument of the complex number. If $\imag{\lambda}>0$, we have that $0 <  \arg \left( c^2 + 4 \epsi(1+\lambda) \right) \leq \pi$ and so $0 < \sin \left( \frac{1}{2}  \arg \left( c^2 + 4 \epsi(1+\lambda) \right) \right) \leq 1$. In other words, if $\imag{\lambda} >0$, then so is $\imag{\mu_-^u}$. The same calculation shows that if $\imag{\lambda}<0$, then $\imag{\mu_-^u}<0$ as well. Similarly, the imaginary part of $\mu_+^s$ has the opposite sign to that of $\imag{\lambda}$. Thus, we have shown that there are no connections possible on this $\eta$ chart. 

Essentially the same calculation shows that there is no connection between $(\mu_-^u)^{-1}$ and $(\mu_+^s)^{-1}$ on the $\tau$ chart and {\color{black} it is worth noting explicitly that the above calculation is independent of the real part of the spectral parameter, and} so we conclude that there are no eigenvalues with non-zero imaginary part {\color{black} (i.e. any eigenvalues must be real)}. Note that this calculation is independent of the continuous spectrum and so in order to conclude stability we will need to take the continuous spectrum into account. 

\subsubsection{Real eigenvalues}\label{subsec:realevals} 

To show that there are no real, positive eigenvalues when $c \geq 2\sqrt{\epsi}$ we proceed as in \cite{rmckrtj12}, although here we avoid the formal machinery discussed therein. Recalling \cref{eq:riccati}, if $\lambda$ is real and positive and if $c \geq 2 \sqrt{\epsi}$, we are looking for a heteroclinic connection on $\R P^1 \approx S^1$, the unit circle. The key idea is to evaluate the Riccati equation on the unit circle at the point $\mu_+^s$. We have the following: 
\begin{equation}\label{eq:key}
\eta' \large|_{\eta = \mu_+^s} =  \frac{1}{\epsi}\left(\lambda + 1 -  2 \hat{u} \right) - \frac{c}{\epsi} \mu_+^s - (\mu_+^s)^2 = \frac{1}{\epsi}(2 - 2 \hat{u})\,, 
\end{equation}
noting here that this is independent of $\lambda$ and strictly positive if $c\geq 2 \sqrt{\epsi}$ (actually for all $c>0$ but if $c\leq 2 \sqrt{\epsi}$, then $\mu_+^s$ is no longer real for all real non-negative values of $\lambda$). 

Next, for each $\lambda \geq 0$, denote the solution on $\R P^1 \approx S^1$ decaying to $\mu_-^u$ by $\ell(z)$. We observe that if $\lambda$ is {\em not} an eigenvalue we have that $\displaystyle \lim_{z \to + \infty} \ell(z) = \mu_+^u(\lambda)$ (or $\mu_+^{sc}(\lambda)$ if $\lambda \leq 1$ to remain consistent with earlier notation). The implication of \cref{eq:key} is the following: 
\begin{proposition}\label{prop:real}
Suppose that $\ell(z)$ crosses $\mu_+^s$ $N_j$ times for some fixed values $\lambda_j$, $j = \left\{ 1, 2 \right\}$. Then, the number of eigenvalues in the interval $(\lambda_1, \lambda_2)$ is equal to $|N_1 - N_2|$. 
\end{proposition}
\begin{proof} Suppose, without loss of generality, that $N_1 = N_2+1$. \Cref{eq:key} means that $\ell(z)$ can only cross $\mu_+^s$ in one direction. This, combined with the previous observations about the limit of $\ell(z)$ for $\lambda$ not an eigenvalue, means that there must be a $\lambda \in (\lambda_1, \lambda_2)$ where $\displaystyle \lim_{z \to +\infty} \ell(z) = \mu_+^s$. Notably, this is the definition of an eigenvalue (Definition \ref{def:eigess}).  Further, the fact that $ \ell(z)$ can cross $\mu_+^s$ in only one direction means that for each eigenvalue $\lambda \in (\lambda_1, \lambda_2)$, the difference $|N_1 - N_2|$ must increase by one. 
\end{proof}

Given Proposition \ref{prop:real}, it suffices to show that there are no crossings for $\lambda$ on the positive real line of $\ell(z)$ as $z$ ranges over $\R$ for $c \geq 2 \sqrt{\epsi}$. If $\lambda = 0$, we have that \cref{eq:system} is the equation of variations along $\hat{u}(z)$ in the phase plane. Thus, the solution $\ell(z)$ is simply the (unit) tangent vector to the curve $(\hat{u}(z), \hat{u}'(z))$ in the phase plane. As $z$ ranges over $\R$ it is obvious that the tangent vector to the curve is never parallel to the eigenvector at positive infinity $(1, \mu^s_+)$.  
Next, for $\lambda \gg 1$ we observe that \cref{eq:system} is hyperbolic, so $\ell(z) \sim \mu_-^u$, the steady state solution. Thus, there are no crossings for $\lambda \gg 1$. This completes the proof that there are no eigenvalues on the positive real line and the proof of spectral stability of the positive travelling waves in the F-KPP equation.   

\begin{remark}
\label{rem:dirty}
To the best of our knowledge this is a new  proof of the absence of eigenvalues {\color{black} with positive real part, and nonzero imaginary part} of the linearised operator about travelling waves in the F-KPP equation of speed $c \geq 2 \sqrt{\delta}$.

\end{remark}

\section{Travelling Waves in a Keller-Segel Model} \label{sec:keller-segel}

We now turn our attention to the application of the techniques from \Cref{sec:fkpp} to a system of PDEs with one spatial and one temporal independent variable, and more than one dependent variable. We focus on the parameter regime of \cref{eq:ks-pde-full} wherein explicit solutions can be found to the travelling wave equation \cref{eq:ks-travellingode}, and so for the remainder of this section, we set $\ve =0$  for unless otherwise specified.

Setting $\ve = 0$, \cref{eq:ks-travellingpde} becomes
\begin{equation} \label{eq:keller-segel-pde-tw}
\begin{split}
u_{\tau} & = cu_z - \alpha w\,, \\
w_{\tau} &= \delta w_{zz} - \beta \left( \frac{w u_z}{u}  \right)_z  + cw_z\,.
\end{split}
\end{equation}

As before, a travelling wave solution will be a stationary solution $(\bar{u}(z),\bar{w}(z))$ to \cref{eq:keller-segel-pde-tw}. In \cite{FC00}, an explicit solution is given:
\begin{equation} \label{eq:keller-segel-sols}
\begin{split}
\bar{u}(z) & = \left( u_r^{-1/\gamma} + \sigma e^{-c(z+z^*)/\delta} \right)^{-\gamma}\,, \\
\bar{w}(z) & = e^{-c(z+z^*)/\delta}[\bar{u}(z)]^{\frac{\beta}{\delta}}\,,
\end{split}
\end{equation}
with 
$$ \gamma = \frac{\delta}{\beta-\delta} >0, \quad \sigma = \frac{\alpha(\beta-\delta)}{c^2} >0,$$
and $z^*$ an integration constant coming from the translational invariance of the travelling wave solutions (owing to the fact that \cref{eq:keller-segel-pde-tw} is autonomous). Without loss of generality we set $z^* =0$. We remark that $u_r$ is the asymptotic limit of the chemical attractant $u$ as $x \to \infty$ and, without loss of generality, as in \cite{HvHP14}, we set $u_r = 1$. See \Cref{fig:ks-sols} for a plot of the solutions $\bar{u}(z)$ and $\bar{w}(z)$ with explicit parameter values. 
 
\begin{figure} 
\includegraphics{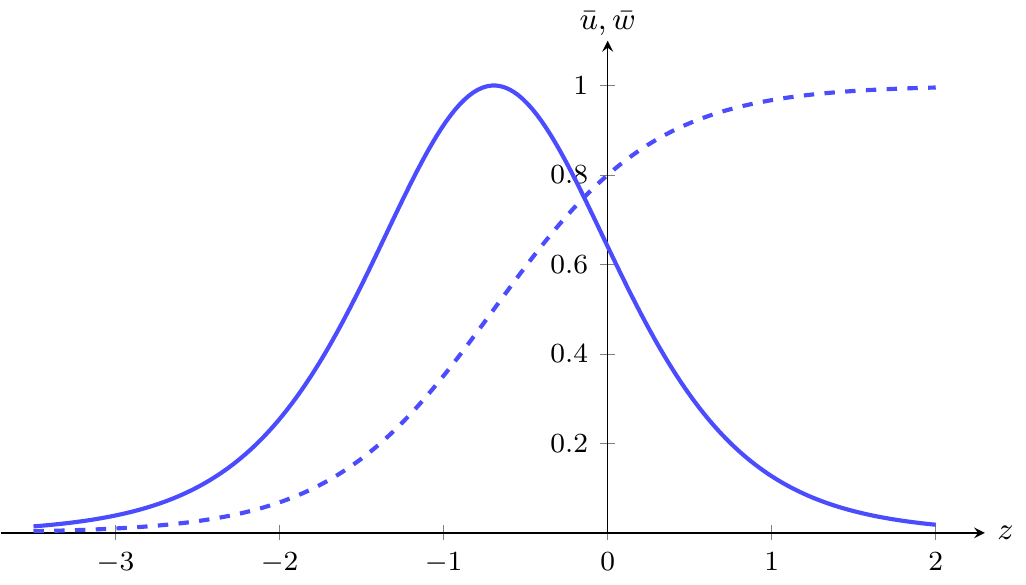}\caption{A plot of the explicit solutions given in \cref{eq:keller-segel-sols} for parameter values $\delta = 1$, $\alpha = 1$, $\beta = 2$, $c = 2$. The dashed front profile is $u(z)$, while the solid pulse is $w(z)$.} \label{fig:ks-sols}
\end{figure}
 
\subsection{{\color{black} The spectral problem}} \label{sec:ksspec} 
The steady state solutions in \cref{eq:keller-segel-sols} (using $' := \frac{d}{dz}$ as before) solve the nonlinear ODEs 
\begin{equation}\label{eq:keller-segel-twode}
\begin{split}
0 & = c u' - \alpha w, \\ 
0 & = \delta w'' + \frac{\alpha \beta}{c} \left( \frac{u' w^2}{u^2} - \frac{2ww'}{u} \right) + c w'.
\end{split}
\end{equation}
Formally, the linearisation of \cref{eq:keller-segel-pde-tw} about the steady state solution $(\bar{u}, \bar{w})$ is given as (dropping the bars for notational convenience) 
\begin{equation}
\begin{pmatrix} p \\ q \end{pmatrix}_ t =  \cL \begin{pmatrix} p \\ q \end{pmatrix}\,,
\end{equation}
where $\cL $ is defined as the following linear operator: 
$\cL :=
\begin{pmatrix}
c \partial_z & -\alpha \\ 
\cL_p & \cL_q \end{pmatrix}
$,
where
\begin{equation}
\begin{split}
\cL_p & := -\frac{\beta  w}{u} \partial_{zz} + \left( \frac{2 \beta  w u'}{u^2}-\frac{\beta  w'}{u} \right) \partial_z +  \left(\frac{\beta  w u''}{u^2}-\frac{2 \beta  w \left(u'\right)^2}{u^3}+\frac{\beta u' w'}{u^2} \right), \\
\cL_q &:= \delta \partial_{zz} + \left(c-\frac{\beta  u'}{u} \right) \partial_z + \left( \frac{\beta  \left(u'\right)^2}{u^2}-\frac{\beta  u''}{u} \right).
\end{split}
\end{equation}

We seek $\lambda \in \C$ for which $\cL - \lambda \I$ is not invertible in some appropriate Banach space. Here, $\cH^1_\nu (\R) \times \cH^1_\nu(\R)$ will suffice, for an appropriately chosen weight $\nu$. For the time being, we set $\nu =0$ and just consider $\cH^1(\R) \times \cH^1(\R)$. The operator $\cL - \lambda \I$ is equivalent to the operator $\cT(\lambda) := \frac{d}{dz} - \A(z,\lambda)$ on the space $\cH^1(\R) \times \cH^1(\R) \times L^2(\R)$ where $\A(z,\lambda)$ is given as 
$$ \A(z,\lambda) := \begin{pmatrix} \frac{\lambda}{c} & \frac{\alpha}{c} & 0 \\ 0 & 0 & 1 \\ \cA & \cB & \cC \end{pmatrix}\,, $$ 
with 
\begin{equation*}
\begin{split}
\cA &:=\left( \frac{\beta  w}{c^2 \delta  u} \right) \lambda^2 + \left( \frac{\beta  w'}{c \delta  u}-\frac{2 \beta  w u'}{c \delta  u^2} \right) \lambda + \frac{2 \beta  w \left(u'\right)^2}{\delta  u^3} -\frac{\beta  w u''}{\delta  u^2} -\frac{\beta  u' w'}{\delta  u^2}  \,, \\ 
\cB & := \left( \frac{\alpha  \beta  w}{c^2 \delta  u}+\frac{1}{\delta } \right) \lambda + \frac{\alpha  \beta  w'}{c \delta  u} -\frac{2 \alpha  \beta  w u'}{c \delta  u^2} +\frac{\beta  u''}{\delta  u}-\frac{\beta 
   \left(u'\right)^2}{\delta  u^2} \,, \\ 
\cC & := -\frac{c}{\delta }+\frac{\alpha  \beta  w}{c \delta  u}+\frac{\beta  u'}{\delta  u} \,.
\end{split}
\end{equation*}
That is, we are looking for solutions in $\cH^1 \times \cH^1 \times L^2$ to the linear, non-autonomous ODEs 
\begin{equation} \label{eq:ks-linearised}
\begin{pmatrix} p \\ q \\ r \end{pmatrix}' = \begin{pmatrix} \frac{\lambda}{c} & \frac{\alpha}{c} & 0 \\ 0 & 0 & 1 \\ \cA & \cB & \cC \end{pmatrix}  \begin{pmatrix} p \\ q \\ r \end{pmatrix}. 
\end{equation}

Observing that the solutions given in \eqref{eq:keller-segel-sols} satisfy 
$$ \lim_{z \to -\infty} (u,w,u',w') = (0,0,0,0) \quad \textrm{and} \quad \lim_{z \to \infty} (u,w,u',w') = (1,0,0,0)\,,$$ 
and that 
$$ u' = \frac{\alpha}{c} w, \quad \textrm{and} \quad \lim_{z \to -\infty} \frac{w}{u} = \frac{c^2}{\alpha(\beta-\delta)}\,,$$ 
we have that the limits as $z \to \pm \infty$ of $\cA$, $\cB$ and $\cC$ denoted $\cA_\pm$, $\cB_\pm$ and $\cC_\pm$, respectively, are
$$ \cA_+ = 0\,, \quad \cB_+ = \frac{\lambda}{\delta}\,, \quad \cC_+ = - \frac{c}{\delta} $$ 
and 
$$
\cA_-  =  \frac{\beta \lambda^2}{\alpha  \delta  (\beta -\delta )} -\frac{\beta  c^2 \lambda }{\alpha  \delta  (\beta -\delta )^2} \,, \quad
\cB_- = \left( \frac{2\beta - \delta}{\delta  (\beta -\delta )} \right) \lambda - \frac{c^2 \beta}{\delta(\beta-\delta)^2} \,, \quad
\cC_- = \frac{c (\beta +\delta )}{\delta  (\beta -\delta )} \,.
$$
We denote $\displaystyle \lim_{z\to\pm\infty}\A(z,\lambda)$ by $\A_\pm(\lambda)$.

\subsection{The continuous spectrum}\label{sec:kscont}
We have that the continuous spectrum (defined as the values of $\lambda \in \C$ for which the signature of $\A_+(\lambda)$ is not equal to the signature of $\A_-(\lambda)$) is bounded by the so-called dispersion relations: the values of $\lambda \in \C$ such that either $\A_+(\lambda)$ or $\A_-(\lambda)$ has a purely imaginary eigenvalue. The dispersion relations are
$$
\lambda = - \delta k^2 + i c k\,, \quad \text{and} \quad \lambda = i c k\,,
$$ 
where $ik,  (k \in \R$)  is the purely imaginary eigenvalue of $\A_+$,  and (implicitly):
\begin{equation} \label{eq:dispaminus}
\begin{split}
\frac{-\lambda^2}{c\delta} & + \left( -\frac{k^2}{c}+i k \left(\frac{1}{\delta} - \frac{1}{\beta -\delta}\right) \right) \lambda - \\ 
 & \frac{c k^2 (\beta +\delta )}{\delta  (\beta-\delta )} + i \left( \frac{\delta  k^3 (\beta -\delta )^2-\beta  c^2 k}{\delta  (\beta -\delta )^2} \right) = 0\,,
\end{split}
\end{equation}
where $ik,  (k \in \R$) is the purely imaginary eigenvalue of $\A_-$. 
We remark that as $\lambda$ only enters \cref{eq:dispaminus} quadratically, an exact expression can be found for it in terms of the other parameters: 
\begin{equation}\label{eq:lampm}
\lambda_\pm := \frac{-\delta(\beta-\delta)k^2 + i c(\beta-2 \delta)k \pm \sqrt{ \Delta} }{2(\beta-\delta)} \,,
\end{equation}
where the discriminant $\Delta$ is given as
\begin{equation}
\Delta:= \left( \delta ^2 (\beta -\delta )^2 \right) k^4 + \left( \beta  c^2 (4 \delta -5 \beta )\right)k^2 + i \left( 2\beta c\delta(\beta-\delta) k^3 - 4 \beta  c^3 k \right). 
\end{equation}

The entire imaginary axis is one of the dispersion relations (and hence contained in the continuous spectrum, $\sige(\cL)$), and, in general, there are points in the continuous spectrum with real part $\lambda>0$, see \Cref{fig:keller-segel-cont-spec} for an illustration. 
%
%To see this analytically, solve \cref{eq:dispaminus} for $\lambda=iy,$ with $y \in \R$ in terms of $k$ and $y$.   There will be five solutions. One at $(k,y) = (0,0)$ and the other four at the zeros of the equations $(k^2 + k_*^2)(y^2 + y_*^2) = 0$ and $(k^2 - k_*^2)(y^2 - y_*^2) = 0$ with 
%$$ 
%k_*^2 = \frac{c^2 \left(\sqrt{\beta  (5 \beta -4 \delta )}-\beta \right)}{2 \delta  (\beta -\delta )^2}
%$$ 
%and 
%$$ 
%y_*^2 = \frac{c^4 \left(\sqrt{\beta  (5 \beta -4 \delta )}+\beta \right) \left(\sqrt{\beta  (5 \beta -4 \delta )}+\beta -2 \delta
%   \right)^2}{8 \delta  (\beta -\delta )^4}
%$$ 
%
%
%One can then verify (though this is time consuming, even using a computer software package) that 
%the derivative $\displaystyle \frac{d\lambda_+}{dk}$ of $\lambda_+$ as a function of $k$, evaluated at the point $k_*$ has a non-zero real part. This means that at this point (i.e. $\lambda = i y_*$), the dispersion relation is entering the right half plane. One can also show that the same thing is happening for $\lambda = -iy_*$ and $k = -k_*$. 

We also note that, as in the F-KPP case, the dispersion relations break up the spectral plane into distinct regions. With a slight abuse of notation, we call region to the right of the continuous spectrum $\Omega_1$. That is: 
$$\Omega_1 := \{ \lambda | \re{\lambda} > \re{\zeta} \  \forall \zeta \in \sige(\cL) \}.$$ 
There are five more regions in the complex plane where the signature of $\A_+(\lambda)$ is the same as that of $\A_-(\lambda)$. The two that are bounded we will denote by $\Omega_4$ and $\Omega_5$, and the three unbounded ones will be denoted $\Omega_2$ (containing an unbounded region of the negative real axis), $\Omega_3$ and $\Omega_6$. The continuous spectrum will be the remaining part of the complex plane: $\sige := \C \setminus \bigcup \Omega_j$. \Cref{fig:keller-segel-cont-spec} shows a plot of the dispersion relations, the regions $\Omega_j$ and the continuous spectrum for explicit choices of the parameter values $\alpha$, $\beta$, $c$ and $\delta$. 

\begin{remark}\label{rem:absspec}
We remark that it is {\em not} possible to weight the continuous spectrum completely into the left half plane. This agrees with known results \cite{nagai1991travelling} about such travelling waves, and suggests the presence of so-called absolute spectrum in the right half plane. Numerically, we were able to (for the parameter values used) determine that the absolute spectrum in the right half plane was contained in a small region $\cR := [0,0.3] \times [4i, -4i] \setminus B_{0.01}(0)$, where $B_{0.01}(0)$ is the ball of radius $r=0.01$ about the origin. See \Cref{fig:ks-absolute}. As some points in the absolute spectrum will coincide with branching points of the Evans function, we generally avoid computing the Evans function in this region. We leave the precise calculation of the absolute spectrum as well as a full spectral analysis of travelling waves in these Keller-Segel models for future work. 
\end{remark}

\begin{figure}
\includegraphics[scale=0.75]{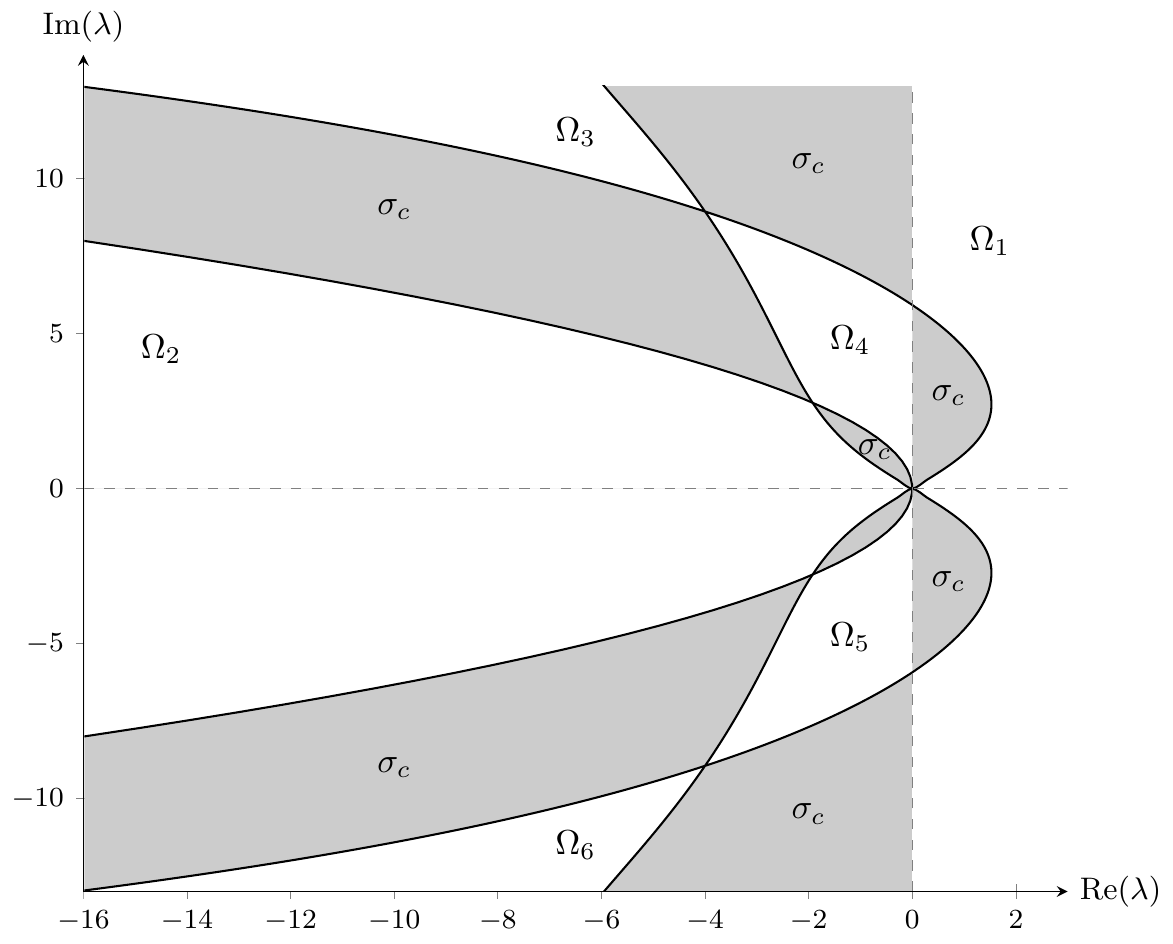}\caption{The continuous spectrum for the linearised operator $\cL$ about the waves in \cref{eq:keller-segel-sols}. The imaginary axis is included in $\sige(\cL)$. It is also evident that some points in the right half of the complex plane are in $\sige(\cL)$. The parameter values for this figure are the same as in \Cref{fig:ks-sols}.}  \label{fig:keller-segel-cont-spec}
\end{figure}

\begin{figure}
\includegraphics[trim= 0cm 10cm 0cm 10cm ,clip, scale=0.75]{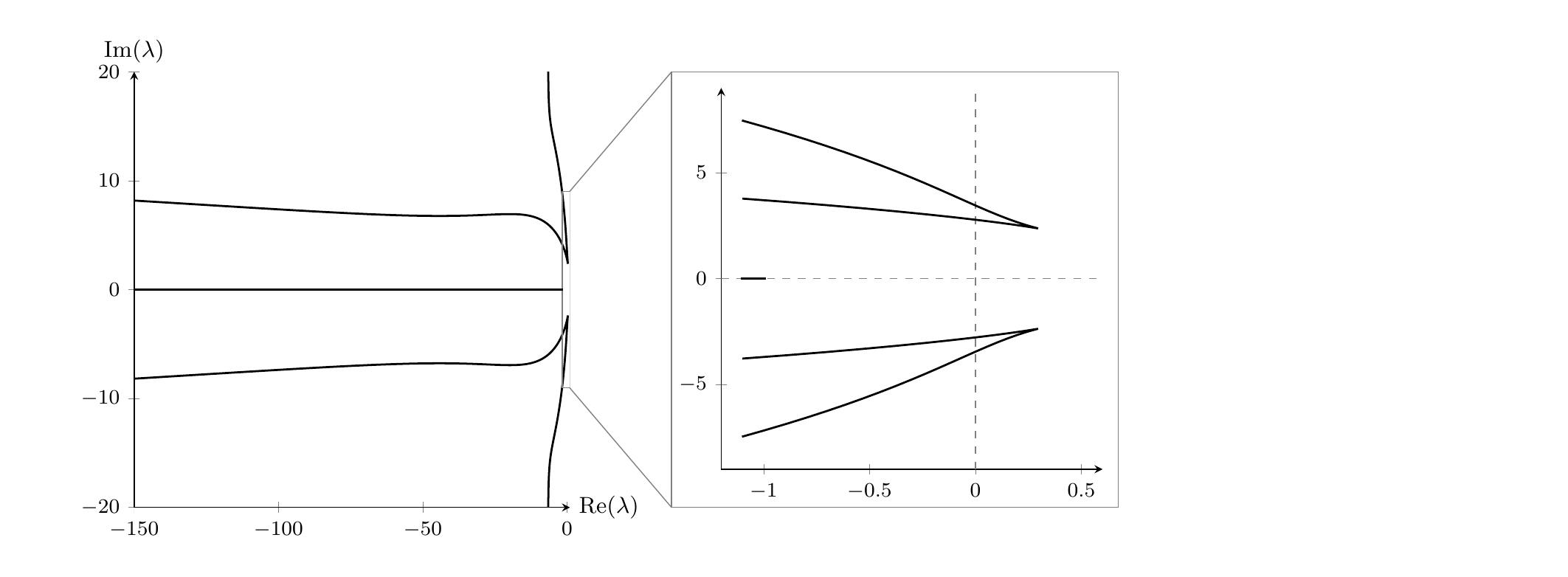}\caption{The (numerically determined) boundary of the absolute spectrum of $\cL$. It clearly contains points in the right half plane, though these points are reasonably far from the origin. The parameter values used were the same as in \Cref{fig:keller-segel-cont-spec}.}\label{fig:ks-absolute}  
\end{figure}

%\begin{figure}[scale=0.75]
%\begin{tikzpicture}
%\begin{axis}[xmin=-150, xmax=1, ymin=-20, ymax=20, xlabel={${\rm Re}(\lambda)$}, ylabel={${\rm Im}(\lambda)$}, width=6cm, height=6cm]
%\addplot[color=black, thick] plot file {ks-absolute-spec-top.txt};
%\addplot[color=black, thick] plot file {ks-absolute-spec-bottom.txt};
%\draw[color=gray] (axis cs:-1.7,-9) rectangle (axis cs:1,9);
%\end{axis}
%
%\begin{axis}[xshift=230, yshift=15, xmin=-1.2, xmax=0.6, ymin=-9, ymax=9, height=5.25cm, width=5.25cm]
%\draw[gray,dashed] (axis cs:0,-9) -- (axis cs:0,9);
%\draw[gray,dashed] (axis cs:-1.2,0) -- (axis cs:0.6,0);
%\addplot[color=black, thick] plot file {ks-absolute-spec-top-zoom.txt};
%\addplot[color=black, thick] plot file {ks-absolute-spec-bottom-zoom.txt};
%\end{axis}
%
%\draw[color=black,thick] (0,3) -- (5.9,3);
%\draw[color=black,thick] (8.35,3.15) -- (8.7,3.15);
%
%
%\draw[color=gray] (7.4,0) rectangle (13.55,6);
%\draw[color=gray] (7.4,0) -- (5.995,1.65);
%\draw[color=gray] (7.4,6) -- (5.995,4.35);
%\end{tikzpicture}
%
%\end{figure}

\subsection{Eigenvalues}\label{sec:kseig}
For $\lambda \in \C\setminus \sige$ we have that $\A_\pm(\lambda)$ are hyperbolic and we again denote the stable and unstable subspaces of $\A_\pm(\lambda)$ as $\xi^s_\pm(\lambda)$ and $\xi^u_\pm(\lambda)$ (or as $\xi^{s,u}_\pm$ where convenient). Just as in \Cref{prop:geometric}, we have that for $\lambda \in \C \setminus \sige$, the existence of a solution to \cref{eq:ks-linearised} decaying to 0 as $z \to \pm \infty$ puts a geometric constraint on the direction of decay. That is: 

\begin{proposition}\label{prop:ks-eigenvalue}
For $\lambda \in \C\setminus \sige$, if $( p, q, r )$ is a solution to \cref{eq:ks-linearised} such that $( p, q , r) \in \cH^1(\R) \times \cH^1(\R) \times L^2(\R)$, then 
\begin{equation}
\lim_{z \to -\infty} \begin{pmatrix} p \\ q  \\ r \end{pmatrix} \to \xi^u_- \quad \text{and} \quad
\lim_{z \to \infty} \begin{pmatrix} p \\ q \\ r  \end{pmatrix} \to \xi^s_+ \, .
\end{equation}
That is, $( p, q, r )$ decays to the stable subspace $\xi^s_+$ of $\A_+(\lambda)$ as $z \to + \infty$ and the unstable subspace $\xi_-^u$ of $\A_-(\lambda)$ as $z \to -\infty$. 
\end{proposition}
Again, see \cite{jones84, kapitula2013spectral, pegowein92} for proofs of this proposition. We call a $\lambda$ for which such a solution exists a {\em (temporal) eigenvalue}, with {\em eigenfunction} $\begin{pmatrix} p \\  q \end{pmatrix}$.  Just as in the F-KPP case, we have that eigenvalues are not possible for all values of $\lambda \in \C \setminus \sige$. In particular, if $\lambda \in \Omega_2$, the unstable subspace of $\A_-(\lambda)$ is zero-dimensional and, hence, the kernel of $\cT(\lambda)$ (or equivalently $\cL$)  is empty by \Cref{prop:ks-eigenvalue}. Further, as we will be primarily concerned with spectral stability, and the regions $\Omega_3, \Omega_4, \Omega_5$ and $\Omega_6$ are all contained in the left half of the complex plane, we again focus our attention on $\lambda \in \Omega_1$, where $\A_+(\lambda)$ will have a one-dimensional stable subspace and where $\A_-(\lambda)$ will have a two-dimensional unstable subspace.

The Evans function in this case is set up similarly. The main difference is that now we have a two-dimensional subspace at $-\infty$. Letting $\Xi^u$ and $\Xi^s$ denote the unstable and stable manifolds respectively, we have that $\Xi^s(z; \lambda)$ is a (complex) line bundle (over $\R$) again while $\Xi^u(z;\lambda)$ will be a complex vector bundle of rank 2. 

We let
$$ \bw^s(z;\lambda) = \begin{pmatrix} w^s_1(z;\lambda) \\ w^s_2(z;\lambda) \\ w^s_3(z;\lambda) \end{pmatrix} \quad \textrm{and} \quad 
\bw^u_{j}(z;\lambda) = \begin{pmatrix} w^u_{j,1}(z;\lambda) \\ w^u_{j,2}(z;\lambda) \\ w^u_{j,3}(z;\lambda)\\  \end{pmatrix}\,,  \quad j = 1,2
$$
be a triple of solutions, with $w^s(z;\lambda) \in \Xi^s$ and $w^u_j(z;\lambda)$ a pair of linearly independent solutions to \cref{eq:ks-linearised} in $\Xi^u$ and define the Evans function:
$$ 
D(\lambda) :=\det \begin{pmatrix}  w^u_{1,1}(0;\lambda) & w^u_{2,1}(0;\lambda) &  w^s_1(0;\lambda) \\
  w^u_{1,2}(0;\lambda) & w^u_{2,2}(0;\lambda)&  w^s_2(0;\lambda)\\
   w^u_{1,3}(0;\lambda) & w^u_{2,3}(0;\lambda) & w^s_3(0;\lambda)\,.
  \end{pmatrix}
$$ 

Just as in the F-KPP case, it is clear that $\lambda \in \Omega_1$ is an eigenvalue if and only if $D(\lambda) = 0$. 
\subsection{The Riccati Equation} \label{sec:ksriccati}
Because we are interested in the evolution of subspaces under the flow of a linear ODE, rather than the behaviour of explicit solutions to \cref{eq:ks-linearised}, it is natural to look at how subspaces evolve under the flow described in \cref{eq:ks-linearised}. Since we have a one-dimensional stable subspace $\xi^s_+$ as $z \to + \infty$, we need to understand how the flow from \cref{eq:ks-linearised} leads to a flow on the set of one-dimensional subspaces in $\C^3$, i.e., the complex projective plane $\C P^2$. Likewise, since we have a two-dimensional unstable subspace $\xi^u_-$ as $z \to -\infty$, we need to translate the flow from \cref{eq:ks-linearised} to a flow on the space of two-dimensional subspaces in $\C^3$. This space is called the {\em complex Grassmannian of two planes in three space} and is denoted $Gr(2,3)$. The associated flow on $\C P^2$ or $Gr(2,3)$ is called the Riccati equation. We obtain an expression on each chart of $\C P^2$ or $Gr(2,3)$ by simply differentiating the defining relations (these are called the {\em Pl\"ucker} relations for Gr(2,3)) for each coordinate on each chart.

For $\C P^2$ this is done in the following way (totally analogous to the $\C P^1$ case). A line in $\C^3$ is determined by a triple of numbers $[p:q:r]$ not all zero and subject to the fact that for any complex number $\zeta$, the triple $[\zeta p: \zeta q : \zeta r]$ represents the same line as $[p:q:r]$. Thus, for example, we can write down all the lines where $q \neq 0$ as 
$[\eta_3 : 1: \eta_4]$. Here, $\eta_3 := \dfrac{p}{q}$ and $\eta_4 := \dfrac{r}{q}$. Differentiating and using \cref{eq:ks-linearised} leads to an expression for the Riccati equation on this chart: 
\begin{equation}\label{eq:ric-1}
\begin{split}
\eta_3' & = \frac{\lambda}{c} \eta_3 + \frac{\alpha}{c} - \eta_3 \eta_4 \\ 
\eta_4' & = \cA \eta_3 + \cB + \cC \eta_4 - \eta_4^2.
\end{split}
\end{equation}
Again, the stable subspace of $\A_+(\lambda)$ will be a point in this chart (usually, if not, use another chart), with a one complex dimensional stable manifold, evolving under the Riccati equation. In this chart of $\C P^1$, we will denote such a solution as $[\eta^s_3(z;\lambda):1: \eta^s_4(z;\lambda)]$. 

For $Gr(2,3)$, we use the standard Pl\"ucker embedding of $Gr(2,3) \to \C P^3$. For a pair of vectors in $\C^3$, $\bv = (v_1,v_2,v_3)$ and $\bw = (w_1,w_2,w_3)$, we have that $\bv$ and $\bw$ are linearly independent (i.e. they span a two-plane) provided that not all of $ K_{i,j} := (v_i w_j - v_j w_i)$ for $1\leq i < j\leq 3$ are zero. This gives us a triple $(K_{12},K_{13},K_{23})$ that must not be all zero if $\bv$ and $\bw$ span a plane. Further, the plane spanned by $\zeta_1 \bv $ and $\zeta_2 \bw$ for $\zeta_{1,2} \in \C$ will be the same as that spanned by $\bv$ and $\bw$ and will produce the triple $\zeta_1 \zeta_2 (K_{12},K_{13},K_{23})$. It is thus clear that we can represent a two-plane in three-space as a triple $[K_{12}:K_{13}:K_{23}]$ in $\C P^2$. 

If $\bv$ and $\bw$ are linearly independent solutions to \cref{eq:ks-linearised}, then by using the product rule, the plane spanned by them in the Pl\"ucker coordinates will solve the linear ODE 
\begin{equation}\label{eq:lingrass}
\begin{pmatrix} K_{12} \\ K_{13} \\ K_{14} \end{pmatrix} ' = \begin{pmatrix} \frac{\lambda}{c} & 1 & 0 \\ \cB & \frac{\lambda}{c} + \cC & \frac{\alpha}{c} \\ -\cA & 0 & \cC \end{pmatrix} \begin{pmatrix} K_{12} \\ K_{13} \\ K_{14} \end{pmatrix} = \A \wedge \A \begin{pmatrix} K_{12} \\ K_{13} \\ K_{14} \end{pmatrix} 
\end{equation}
where the last $\A \wedge \A$ means the exterior product of the matrix $\A$ with itself. 

The idea now is to use the Riccati equation for \cref{eq:lingrass} to write down how the linear flow given by \cref{eq:ks-linearised} behaves on pairs of subspaces. From this perspective it is clear that we have three charts from which to choose for the Pl\"ucker embedding of $Gr(2,3)$ (on which the unstable manifold will be a curve) and we have three for the Pl\"ucker embedding of $Gr(1,3)$ (on which the stable manifold will be a curve). Suppose for concreteness, that $K_{12} \neq 0$. Then by setting $\kappa_5 = -\frac{K_{23}}{K_{12}}$ and $\kappa_6 = \frac{K_{13}}{K_{12}}$, and using \cref{eq:lingrass} we have that $\kappa_5$ and $\kappa_6$ will satisfy the nonlinear ODEs
\begin{equation}\label{eq:ric-2}
\begin{split}
\kappa_5' & = \cA + \left(\cC - \frac{\lambda}{c} \right) \kappa_5 - \kappa_5 \kappa_6\\ 
\kappa_6' & = \cB - \frac{\alpha}{c} \kappa_5 + \cC \kappa_6 - \kappa_6^2 
\end{split}
\end{equation}

The unstable subspace of $\A_-(\lambda)$ will be a point on this chart (usually) and it has a one-dimensional unstable manifold, denoted in coordinates on this chart as $[1: \kappa_6^u(z;\lambda): -\kappa_5^u(z;\lambda)]$. 

All that remains is how to relate $\eta_{3,4}$ and $\kappa_{5,6}$ to $D(\lambda)$. Proceeding as we did in the F-KPP case, suppose that the solution $\bw^s$ stays in the same chart (of $\C P^1$) for all $z$ and that the pair of solutions $(\bw^u_1, \bw^u_2)$ stay on the same chart (of $Gr(2,3)$) for all $z$. By way of example, suppose it is in the two charts for which we have written expressions for the Riccati equation, \cref{eq:ric-1,eq:ric-2}, respectively. Then, in particular, we have that $w^s_2(z;\lambda) \neq 0$ and the matrix $\bW^u_{12} := \begin{pmatrix} w_{1,1}^u(z;\lambda) & w_{2,1}^u(z;\lambda) \\ w_{1,2}^u(z;\lambda) & w_{2,2}^u(z;\lambda) \end{pmatrix}$ is invertible for all $z$ (because we are in the charts where $q \neq0$ and where $K_{12} \neq 0$). Defining
$$
K_{12}^u(z;\lambda) := \det \bW^u_{12} = \begin{vmatrix} w_{1,1}^u(z;\lambda) & w_{2,1}^u(z;\lambda) \\ w_{1,2}^u(z;\lambda) & w_{2,2}^u(z;\lambda)  \end{vmatrix} =  w_{11}^u(z;\lambda) w_{22}^u(z;\lambda) - w^u_{12}(z;\lambda) w^u_{21}(z;\lambda) \neq 0
$$

We have that the matrix
\begin{equation} 
\begin{split}
\begin{pmatrix}  w^u_{1,1}(z;\lambda) & w^u_{2,1}(z;\lambda) &  w^s_1(z;\lambda) \\
  w^u_{1,2}(z;\lambda) & w^u_{2,2}(z;\lambda)&  w^s_2(z;\lambda)\\
   w^u_{1,3}(z;\lambda) & w^u_{2,3}(z;\lambda) & w^s_3(z;\lambda) \end{pmatrix}  & 
   \begin{pmatrix} \left(\bW^u_{12}\right)^{-1}  & 0 \\ 0  & \frac{1}{w_2^2(z;\lambda)} \end{pmatrix} \\   =  
  \begin{pmatrix} 1 & 0 & \eta^s_3(z;\lambda) \\ 0 & 1 & 1 \\ \kappa_5^u(z;\lambda)  & \kappa_6^u(z;\lambda)  & \eta_4^s(z;\lambda) \end{pmatrix} & 
  \end{split}
\end{equation}
is well defined for all values of $z$. Evaluating at $z=0$ and taking determinants gives
\begin{equation}\label{eq:evans}
\frac{D(\lambda)}{K_{12}^u(0;\lambda) w^s_2(0;\lambda)} = \eta_4^s(0;\lambda) - \kappa_6^u(0;\lambda) - \eta_3^s(0;\lambda)\kappa_5^u(0;\lambda)\,.
 \end{equation}
Define the function
 $$
E_{12q}(\lambda) :=  \eta_4^s(0;\lambda) - \kappa_6^u(0;\lambda) - \eta_3^s(0;\lambda)\kappa_5^u(0;\lambda)\,.
 $$
The subscripts indicates that the $q$ coordinate of $\Xi^s$ and the $K_{12}$ coordinate of $\Xi^u$ are both $ \neq0$. 
Since each of the solutions that we are tracking stay in the same chart, $E_{12q}(\lambda) = 0$ if and only if $D(\lambda) =0.$ Again, provided the solutions $\eta^s_{3,4}$ and $\kappa^u_{5,6}$ stay in the same charts, we can use the argument principle to determine the number of zeros $E_{12q}(\lambda)$ has for any prescribed curve in the region $\Omega_1$. 
  
\subsection{Switching charts and extending into the continuous spectrum} \label{sec:ks-charts} Just as in the F-KPP case, should a singularity of the solution of the Riccati equation appear, we can interpret this as the solution leaving the chart. Then we can switch to a different chart by the same method described earlier: namely choose a value $z_0$ for which the solution is not singular, use this as an initial condition on a different chart and evolve the solution on said chart beyond the point of singularity. Then, if desired, one can switch back to the original chart. 

It is also worth noting, that as we are only ever tracking a finite number of solutions to the Riccati equation on compact manifolds, it is always possible to find at least one set of charts (one for $\C P^2$ and one for $Gr(2,3)$), on which all of the solutions we are interested in will remain for all $z$ (though this is not necessarily always one of the canonical charts). That is, it is always possible to choose charts so that the solutions used in the shooting for the Evans function stay bounded for all values of the independent variable. For the parameter values considered in this example, we found that the charts $[\frac{p}{q}: 1: \frac{r}{p}]$, and [$1: \frac{K_{13}}{K_{12}}:\frac{K_{23}}{K_{12}}]$ would suffice for all $\lambda$ with $\re{\lambda} \geq 0$ (excepting a small neighbourhood containing the absolute spectrum see \Cref{rem:absspec}).

This function $E_{12q}(\lambda)$ (or its analog on any pair of charts from $\C P^2$ and $Gr(2,3)$) can naturally be extended into the continuous spectrum. We define `eigenvalues', as in the F-KPP case, not as values of $\lambda$ for which we can find a solution to \cref{eq:ks-linearised} decaying to zero but for which we can find a solution to 
\cref{eq:ks-linearised} decaying in a specific, geometric way. As we vary $\lambda$ across the dispersion relation curves into the continuous spectrum, we can continuously track $\xi^s_+(\lambda)$ and $\xi^u_-(\lambda)$. This gives a straightforward continuation of $E_{12q}(\lambda)$ (or its analogs on other charts) into the continuous spectrum (though not the absolute spectrum).

\subsection{Stability Analysis}\label{ks-stab} 
In this section we numerically establish that there is no point spectrum of the operator $\cL$ with real part between $0$ and $10^7$, except possibly in the region $\cR := [0,0.3] \times [4i, -4i] \setminus B_{0.01}(0)$. We also show that $\lambda = 0$ is an eigenvalue of multiplicity $2$. For this analysis, the parameter values chosen were the same as in \cite{harley2014geometric}, namely $\alpha =1, \beta=2, c= 2,$ and $\delta =1$. 

%First, we note that the presence of the imaginary axis, both in and of itself and as a vertical asymptote to part of the curves in the dispersion relations means that the linearised operator $\cL$ described above is not sectorial, and, hence, we can not infer linear stability from (weighted) spectral stability.  It does appear however, that the operator $\cL$ is spectrally stable in an appropriate weighted space.

%Like the F-KPP travelling waves, in order to shift the continuous spectrum of $\cL$ into the left half plane, not all weights will suffice; there is a bounded range of weights which will shift the continuous spectrum into the left half plane. Exactly determining a formula for the upper and lower bounds of the exponents of the weight functions is technically possible, however it is not particularly illuminating, nor useful for computations, as for any particular set of parameters we can numerically approximate the bounds of the weights that will shift the continuous spectrum into the left half plane. For example for the travelling waves with the parameter values $\alpha = 1, \delta = 1, \beta = 2$, and  $ c = 2$ we found that if we considered values $\nu \in (-2, -0.153)$, then the spectrum of $\cL$ (as an operator on the space $\cH^1_\nu$) was shifted into the left half plane.   

Using the Ricatti Evans functions outlined in this section, we can numerically verify that there are no eigenvalues (in the sense of \Cref{def:eigess}) for a large region in the right half complex plane (out to $|\lambda| < 10^7 $), both within and without of the continuous spectrum. We first compute the Evans function $E_{12q}(\lambda)$ on a spectral curve consisting of the right half of an annulus (including the imaginary axis) with inner radius $r = 4$ and outer radius $r =10^7$. We can visually inspect that there is no winding of the Evans function around the origin, and thus conclude that there are no eigenvalues of the operator $\cL$ in this region,  see \Cref{fig:ks-evans1}.

We next compute the function $E_{12q}(\lambda)$ for $\lambda \in \C$ on the boundary of the half disc of radius $r=4$ shifted to the right by $0.3$ (see \Cref{fig:ks-evans2}). Again, here we can visually inspect that the winding number of the Evans function about the origin is zero, and we conclude that there are no eigenvalues of $\cL$ in this region either. \Cref{fig:ks-evans1,fig:ks-evans2} allow us to conclude that all eigenvalues of the the operator $\cL$ in the right half plane either have norm greater that $10^7$ or else lie in the region $\cR := [0,0.3] \times [4i, -4i] \setminus B_{0.01}(0)$ in the complex plane. 

 \begin{figure}
\includegraphics[scale=0.75]{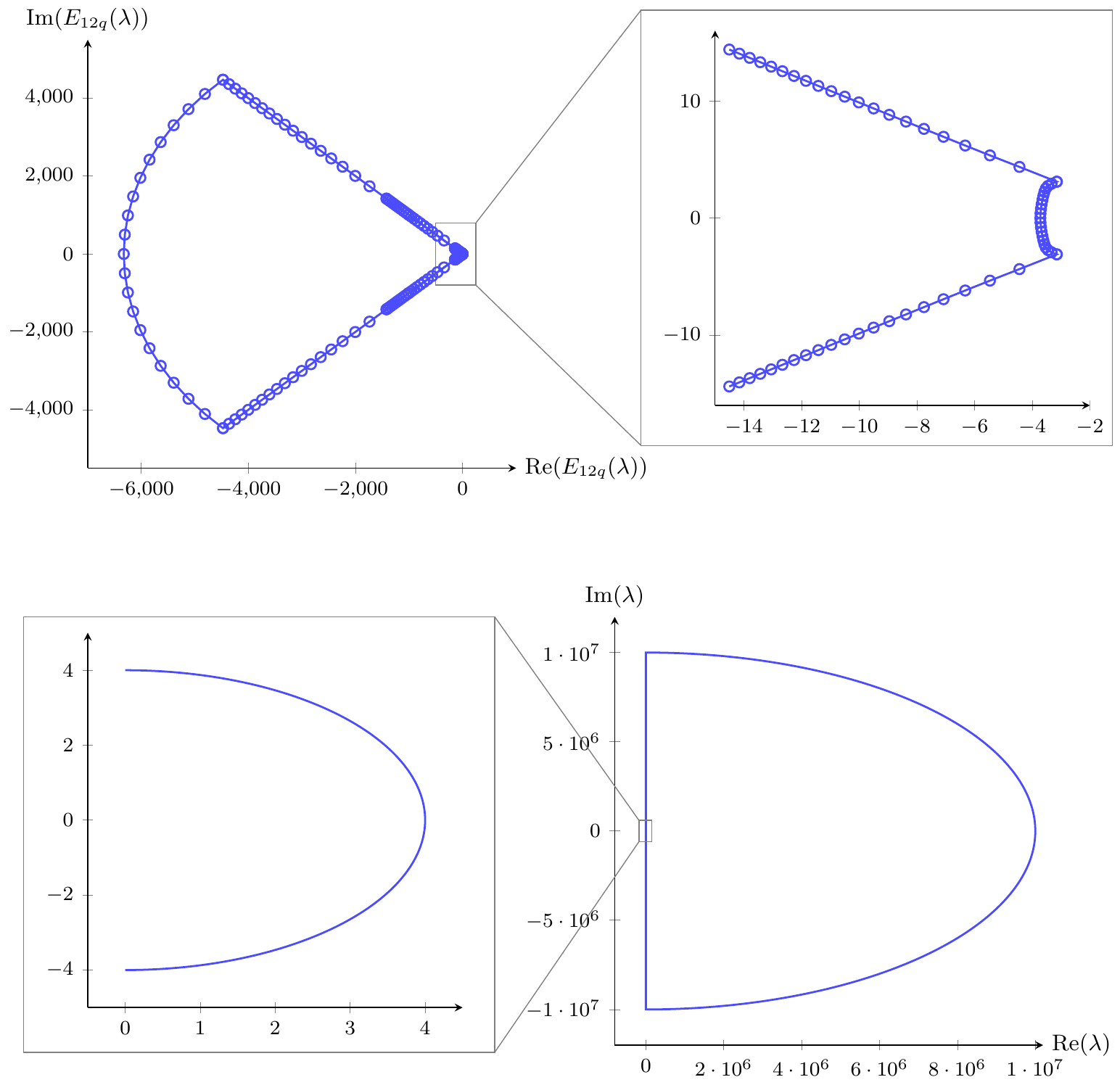}
 \caption{A plot of the function $E_{12q}(\lambda)$ (top) for $\lambda$ on the closed curve on the bottom. We have that $4 <|\lambda|<10^7$. As can be seen in the top right figure, the image of the Evans function clearly does not wind around the origin. We conclude that there are no eigenvalues in this region. The top right picture is a zoomed in plot of the Evans function nearer to 0, while the bottom right picture is a zoomed in plot of the curve in the spectral plane.} \label{fig:ks-evans1}
 \end{figure}
 
   \begin{figure}
\includegraphics[scale=0.75]{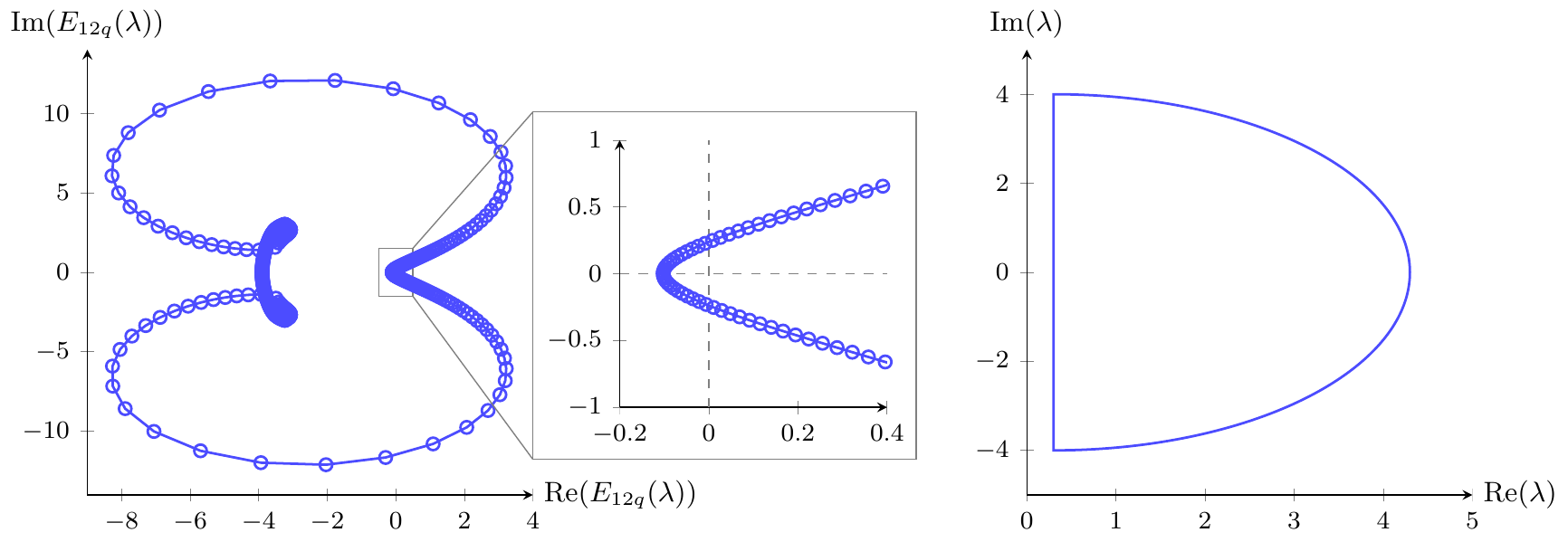}
 \caption{A plot of the function $E_{12q}(\lambda)$ (left) for $\lambda$ on the closed curve on the right. The central inset shows that there is no winding of the Evans function in this region about the origin either.} \label{fig:ks-evans2}
 \end{figure}

In order to evaluate the function $E_{12q}(\lambda)$ reasonably efficiently, one needs to be sufficiently far enough away from the absolute spectrum. For the parameters considered in this manuscript, it was found that the absolute spectrum is not the entire region $[0,0.3] \times [-4i,4i]$, but is bounded away from the origin (see \Cref{fig:ks-absolute}). We were thus able to evaluate the Evans function $E_{12q}(\lambda)$ for $\lambda$ on the boundary of a small disc (radius $r = 10^{-2}$) about the origin. We found that on this boundary the function $E_{12q}(\lambda)$ wound around the origin two times, and so we conclude that $\lambda=0$ is an eigenvalue of multiplicity $2$. See \Cref{fig:ks-evans3,fig:ks-evans4}. 

  \begin{figure}
\includegraphics[scale=0.75]{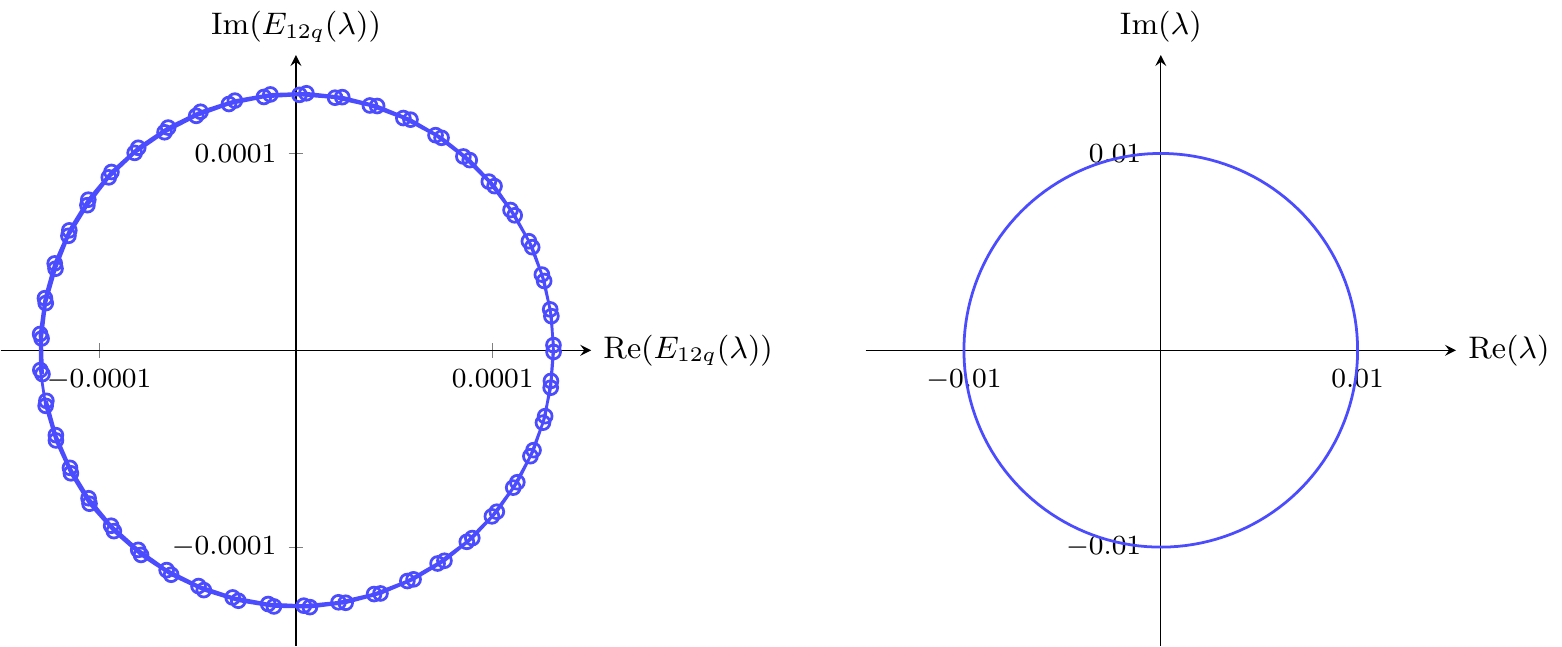}
 \caption{A plot of the function $E_{12q}(\lambda)$ (left) for $\lambda$ on the closed curve on the right. It is clear that the Evans function winds around the origin, suggesting that $\lambda=0$ is an eigenvalue.} \label{fig:ks-evans3}
 \end{figure}

   \begin{figure}
 \includegraphics[scale=1]{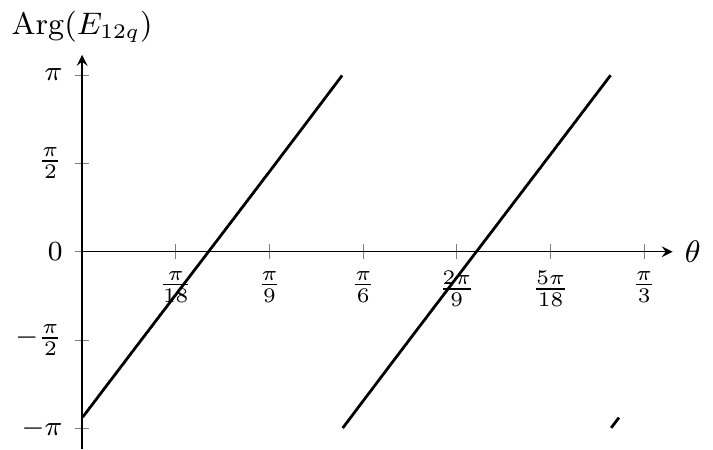} \\
 \caption{A plot of the argument of the function $E_{12q}(\lambda)$ for $\lambda = 10^{-2}e^{2\pi it}$ as $t$ passes through a domain of length 1.The argument goes through a change of $4\pi$ suggesting that $0$ is an eigenvalue of multiplicity $2$.} \label{fig:ks-evans4}
 \end{figure}

\section{Summary of Results and Concluding Remarks} \label{sec:summary} 

We have illustrated how to use the underlying geometry of the spectral problem in order to facilitate computation of the spectrum of a linearised operator about a travelling wave in a PDE with 1+1 independent variables. The geometric interpretation of the Riccati equations allows us to handle the blow-up of solutions to nonlinear ODEs. We have thus used these solutions to develop new Evans functions, and used them to numerically verify the spectral stability of travelling waves in the F-KPP equation, and the absence of eigenvalues in a large region of the complex plane for the the explicit travelling waves in {\color{black} the K-S system}
when $\ve = 0$.  We have also shown in this case that $\lambda =0$ is an eigenvalue of multiplicity $2$. 

The Evans functions we have produced are seemingly very well behaved in comparison to more naive attempts at computing them. They are reasonably easy to compute for large values of the spectral plane, and their winding around the origin can be visually inspected in both the examples that we have shown. Finally, our methods are fairly general, we are able to develop the corresponding Riccati Equations and Evans functions for a general class of non-self adjoint operators, and we can compute the Evans functions for a large set of values in the spectral plane and also, regardless of the dimensions of the stable and unstable subspaces at $\pm \infty$, $\xi^u_-$ and $\xi^s_+$.

\subsection{Summary of stability results}

We have verified that the continuous spectrum of the linearised operator $\cL$, linearised about travelling waves in the F-KPP equation can be weighted to the left half plane,
Further we have explicitly verified that there are no eigenvalues in the sense of \Cref{def:eigess} with $\re{\lambda}\geq0$, in the F-KPP travelling waves with wave speed $c> 2\sqrt{\delta}$. We have provided a new proof of spectral stability of the travelling waves of speed $c > 2 \sqrt{\delta}$ to the F-KPP equation. Since the operator is sectorial, we can therefore confirm linear stability of the F-KPP travelling waves \cite{kato76}.

For the K-S system when $\ve =0$ and for the explicit solutions in \cref{eq:keller-segel-sols} and parameters considered, we were unable to weight the continuous spectrum into the left half plane. This is consistent with known results about the system \cite{nagai1991travelling} and suggests the presence of absolute spectrum with positive real part. The absolute spectrum appears to be bounded away from the origin, and therefore enters the right half plane at some point on the imaginary axis (for the parameter values used in this work, we numerically found this to be between $\pm 2i$ and $\pm4i$, see \Cref{fig:ks-absolute}). It is unclear what effects this has on the dynamics of the travelling waves, and further study is required. 

We have verified that for the linear operator linearised about the Keller--Segel waves $\bar{u}$ and $\bar{w}$ in \cref{eq:keller-segel-sols}, there are no eigenvalues with  $0\leq \re{\lambda} \leq 10^7$ except possibly in the region $\cR := [0,0.3] \times [4i, -4i] \setminus B_{0.01}(0).$ We have also numerically shown that $0$ is an eigenvalue of multiplicity $2$. 

\subsection{Future Work: The K-S system in the case when $\ve \neq 0$}\label{sec:smalldisp}

If we return to \cref{eq:ks-pde-full} and consider $0 < \ve \ll 1$, travelling waves are still known to exist (see for example \cite{wang2013} and the references therein) though no explicit formula for them is known. Further it was shown in \cite{HvHP14} that the travelling wave solutions in this case, say $\bar{u}_\ve(z)$ and, $\bar{w}_\ve(z)$ are perturbations of $\bar{u}$ and $\bar{w}$ from \cref{eq:keller-segel-sols}. One could then linearise around $(\bar{u}_\ve(z), \bar{w}_\ve(z))$, and by computing the asymptotic limits of the functions, their derivatives and appropriate ratios of them, determine the dispersion relations, and subsequently the continuous spectrum of the linearised operator. We conjecture (as is typical in these types of travelling wave examples) that the inclusion of a nonzero diffusion term in the first equation of \cref{eq:ks-pde-full} will lead to the resulting linearised operator being sectorial. In this instance however, we expect to see absolute spectrum in the right half plane, though the impact of this on the explicit dynamics as in the $\ve = 0$ case may not be clear. 

We then aim to repeat the procedure outlined above to numerically investigate whether there were eigenvalues for the linearised system. Numerically finding $\bar{u}_\ve(z)$ and $\bar{w}_\ve(z)$ is a bit time consuming, and as this manuscript was primarily to provide examples illustrating our methods, we have, in the interest of expediency, elected to focus on the model where explicit solutions are known. 

Provided that one can numerically find the solutions $(\bar{u}_\ve, \bar{w}_\ve)$ however, it is not difficult to extend our methods to compute a similar Evans function and determine the presence (or lack thereof) of eigenvalues in the right half plane. The emerging Riccati equations will determine a flow on $Gr(k, 4)$, the Grassmannian of $k$ planes in $\C^4$, where $k$ is determined by the dimensions of the stable and unstable subspaces of the asymptotic end states of the operator for $\lambda$ in the region equivalent to $\Omega_1$ (i.e.\,to the right of the continuous spectrum). Further, the expressions for the Riccati equations are found in much the same way as for the Keller--Segel and F-KPP models, one must just use a different Pl\"ucker embedding for each separate $k$ appearing in the problem. The expressions will require ${{4}\choose{k}} - 1$ variables (the dimension of the range space in the Pl\"ucker embedding), however some of these can be eliminated by the so-called {\em Pl\"ucker relations}, the varietal conditions that the Grassmannian $Gr(4,k)$ must satisfy. In this case, there is at most one of these, and only when $k=2$, but for higher order systems, there can be many more (there will be  ${{n}\choose{k}} - 1 - k(n-k)$ of them for a general system.)

\bibliographystyle{plain}

\bibliography{fisherkpp_mathbio.bib}

\end{document}